\RequirePackage[l2tabu, orthodox]{nag}
\documentclass[reqno]{amsart}
\usepackage[pdftex]{graphicx,color}
\usepackage[pdftex]{hyperref}

\usepackage{amsmath,amssymb,mathrsfs}
\usepackage{amsmath,amsthm,amssymb}

\usepackage{amsfonts}
\usepackage{latexsym}

\usepackage{array}
\usepackage{amsthm}
\usepackage{ascmac}
\usepackage[foot]{amsaddr}
\usepackage{enumerate}

\theoremstyle{definition}
\newtheorem{dfn}{Definition}[section]
\newtheorem{thm}[dfn]{Theorem}
\newtheorem{prop}[dfn]{Proposition}
\newtheorem{lem}[dfn]{Lemma}
\newtheorem{rmk}[dfn]{Remark}

\newtheorem*{dfn*}{Definition}
\newtheorem*{thm*}{Theorem}

\usepackage{graphicx}
\usepackage{setspace}
\usepackage{ascmac}
\usepackage{bm}
\usepackage{dynkin-diagrams}
\usepackage{comment}

\DeclareMathOperator{\SPAN}{span}

\title{\large{Level 2 standard modules for $A^{(2)}_{9}$ and partition conditions of Kanade-Russell}}

\author{Kana Ito}
\date{}
\address{Tokyo Institute of Technology, RIKEN AIP}
\email{ito.k.bn@m.titech.ac.jp}

\begin{document}
	\begin{abstract}
		We give $Z$-monomial generators for the vacuum spaces of
		certain level 2 standard modules of type $A^{(2)}_{\textrm{odd}}$
		with indices 
		running over
		integer partitions.
		In particular,  we give a Lie theoretic interpretation of
		the Rogers-Ramanujan type identities of 
		type $A^{(2)}_{9}$,
		which were conjectured by Kanade-Russell, and proven by Bringmann et al.  and Rosengren.
	\end{abstract}

\maketitle
\allowdisplaybreaks

\section{Introduction}
	\subsection{Rogers-Ramanujan type identities and affine Lie algebras}
		The RR (short for the Rogers-Ramanujan) identities are stated as
			\begin{align*}
				\sum^{\infty}_{n=0}
				\frac{q^{n^{2}}}{(q;q)_{n}}
				=\frac{1}{(q,q^{4};q^{5})_{\infty}},\qquad
				\sum^{\infty}_{n=0}
				\frac{q^{n^{2}+n}}{(q;q)_{n}}
				=\frac{1}{(q^{2},q^{3};q^{5})_{\infty}},
			\end{align*}
		where,  for $n\in\mathbb{Z}_{\ge0}\cup\{\infty\}$,  the Pochhammer symbols are defined as
			\begin{align*}
				(a;q)_{n}
				\coloneqq
				\prod^{n-1}_{i=0}(1-aq^{i}),
				\qquad
				(a_{1},\dotsc,a_{k};q)_{n}
				\coloneqq
				(a_{1};q)_{n}\cdots(a_{k};q)_{n}.
			\end{align*}

		MacMahon and Schur found that the RR identities are equivalent to 
		the RRPT (PT is short for partition theorem) \cite[Chapter 2.4]{S}. 
		The RRPT is stated as;
			
		\begin{description}
			\item[RRPT]
				The number of partitions of $n$ 
				in $R_{i}$ below 
				is equal to
				the number of partitions of $n$ in $T^{(5)}_{i,5-i}$ 
				(i.e., $R_{i}\stackrel{\rm{PT}}{\sim}{T^{(5)}_{i,5-i}}$) for $i=1,2$.
		\end{description}
		Here,  $n$ is a non-negative integer, 
		and an integer partition is an element in the set
			\begin{align*}
				\textrm{Par}\coloneqq
				\{(\lambda_{1},\lambda_{2},\dotsc,\lambda_{l})
				\in(\mathbb{Z}_{>0})^{l}\mid
				l\ge0,\ 
				\lambda_{1}\ge\lambda_{2}\ge\cdots\ge\lambda_{l}\},
			\end{align*}
		and a partition of $n$ is
		$\lambda=(\lambda_{1},\lambda_{2},\cdots,\lambda_{l})\in\textrm{Par}$ such that
		$|\lambda|\coloneqq\lambda_{1}+\lambda_{2}+\cdots+\lambda_{l}=n$.
		The subsets of $\textrm{Par}$,  $R_{1},R_{2}$ and $T^{(N)}_{p_{1},\dotsc,p_{k}}$
		are defined as follows,
		\begin{align}
			R_{1}&\coloneqq\{(\lambda_{1},\dotsc,\lambda_{l})\in\textrm{Par}\mid
			\lambda_{i}-\lambda_{i+1}\ge2
			\textrm{ for }1\le i\le l-1\},\label{RRS1}\\
			R_{2}&\coloneqq\{(\lambda_{1},\dotsc,\lambda_{l})\in R_{1}\mid
			\lambda_{l}\ge2\},\label{RRS2}\\
			T^{(N)}_{p_{1},\dotsc,p_{k}}
			&\coloneqq
			\{(\lambda_{1},\dotsc,\lambda_{l})\in \textrm{Par}\mid
			\lambda_{i}\equiv p_{1},p_{2},\dotsc,p_{k}\textrm{ (mod }N)
			\textrm{ for }1\le i\le l\},\notag
		\end{align}
		where $N\ge2$ and $k$ are positive integers, 
		and $p_{i}$ for $1\le{i}\le{k}$ are non-negative integers.
		For $A,B\subseteq\textrm{Par}$,  let $A\stackrel{\rm{PT}}{\sim}{B}$ signify that
		the number of partitions of $n$ in $A$ is equal to the number of partitions of $n$ in $B$
		for every non-negative integer $n$.
%
%
		For the history and integer partitions,  see \cite{And2} and \cite{S}.

		In 1970's, Lepowsky-Milne \cite[Theorem 5.14]{LM} found similarities
		between the infinite products of the RR identities and
		the characters of the level 3 standard modules
		(i.e., the highest weight  integrable modules)  of type $A^{(1)}_{1}$.
		Furthermore,  Lepowsky-Wilson \cite[Corollary 10.5]{LW3}
		(cf. \cite{LW1}, \cite{LW2}) proved the RR identities
		via the vacuum spaces $\Omega(L)$ of the level 3 standard modules of type $A^{(1)}_{1}$. 
		Since the work by Lepowsky-Wilson,
		there has been an expectation that RR type identities and partition theorems are 
		obtained from the standard modules of affine Lie algebras.

		A well-known example is the Andrews-Gordon identities
		for type $A^{(1)}_{1}$ odd levels,
		and their analogue for even levels \cite{LW4},  \cite{MP}.
		For type $A^{(2)}_{2}$ level 3,
		Capparelli \cite{CapD} made conjectures of partition theorems,
		which were proven by Andrews \cite{And}, Tamba-Xie \cite{TX},  and Capparelli \cite{Cap}.
		Nandi \cite{Nan} made conjectures of partition theorems
		for type $A^{(2)}_{2}$ level 4,
		which were proven by Takigiku-Tsuchioka \cite{Tak}.
 	
		Regarding the cases of $A^{(2)}_{\textrm{odd}}$ level 2,
		Kanade \cite{Kan} and Bos-Misra \cite{BM} gave a Lie theoretic interpretation
		to the RR type identities respectively in the cases of $A^{(2)}_{5}$ and $A^{(2)}_{7}$ level 2.  		Following these results, we focus on the cases of $A^{(2)}_{\textrm{odd}}$ level 2.

	\subsection{The main results}
		The main results are Theorem \ref{MT} and \ref{A29thm}.
		
		Consider the case of $A^{(2)}_{2s-1}$.
		Let
		$[a;q]_{\infty}\coloneqq(a,q/a;q)_{\infty}$.
		The level 2 standard modules are given by
		$\Lambda^{(n)}\coloneqq(\delta_{0,n}+\delta_{1,n})\Lambda_{0}+\Lambda_{n}$
		for $0\le n\le s$ (modulo the Dynkin diagram).
		The principal character of $L(\Lambda^{(n)})$
		for $0\le n\le s$	 is
		\begin{align}\label{Angelica}
			\chi_{A^{(2)}_{2s-1}}(\Lambda^{(n)})
			=
			\frac{(q^{2s+2};q^{2s+2})_{\infty}}{(q^{2};q^{2})_{\infty}}
			\frac{[q^{2n};q^{2s+2}]_{\infty}}{[q^{n};q^{2s+2}]_{\infty}}.
		\end{align}
		The G\"{o}llnitz-Gordon identities (resp.  the RR identities)
		are obtained in the case of
		$A^{(2)}_{5}$ (resp.  $A^{(2)}_{7}$)
		for $L(\Lambda^{(n)})$
		with
		$n\in\mathbb{Z}\backslash2\mathbb{Z}$ (i.e., $n$ is odd).
		Based on the previous works such as \cite{Kan},  \cite{BM},  and \cite{KR},
		we reduce $Z$-monomial spanning sets of the vacuum spaces of
		$L(\Lambda^{(n)})$ with
		$0\le n\le s$ and 
		$n\in\mathbb{Z}\backslash2\mathbb{Z}$
		in the case of $A^{(2)}_{2s-1}$.

	\subsubsection{Type $A^{(2)}_{\textrm{odd}}$ level 2}
		The following are the Dynkin diagrams of
		type $A_{2s-1}$ and type $A^{(2)}_{2s-1}$ respectively \cite[Chapter 4]{Kac}. 
%
%
		\begin{center}    
			\begin{dynkinDiagram}[text style/.style={scale=0.8},
				edge length=0.9cm,
				labels={1,2,s,2s-2,2s-1},
				label macro/.code={\alpha_{\drlap{#1}}}
				]A{**.*.**}
			\end{dynkinDiagram}
			$\hspace{1cm}$
			\begin{dynkinDiagram}[text style/.style={scale=0.8},
				edge length=0.9cm,
				labels={0,1,2,3,4,s-2,s-1,s},
				label macro/.code={\alpha_{\drlap{#1}}}
				]A[2]{odd}
			\end{dynkinDiagram}
		\end{center}

		One of our main results is as follows.
		\begin{thm}\label{MT}
			Let $L=L(\Lambda^{(n)})$
			be a level 2 standard module of type $A^{(2)}_{2s-1}$, with a highest weight vector
			$v_{L}$ for $1\le n\le s$ and $n\in\mathbb{Z}\backslash2\mathbb{Z}$.
			Then the following set spans $\Omega(L)$,
			\begin{align}\label{MT'}
				\{Z_{i_{1}}\cdots{Z_{i_{l}}}v_{L}
				\mid l\ge0,\hspace{1mm}i_{1},\dotsc,i_{l}\in\mathbb{Z},
				\ i_{1}\le\cdots\le{i_{l}}\le{-1}
				\},
			\end{align}
			where $Z_{2i}(\alpha_{1})=Z_{2i}$ and $Z_{2i-1}(\alpha_{s})=Z_{2i-1}$
			for $i\in\mathbb{Z}$.
			See Definition \ref{Zdef} for the definition of the $Z$-operators.
		\end{thm}
		In other words, $\Omega(L)$ has a spanning set of $Z$-monomials
		with indices corresponding to integer partitions,
		consisting of $Z$-operators associated with the edge vertex $\alpha_{1}$
		and the middle vertex $\alpha_{s}$ of the Dynkin diagram of type $A_{2s-1}$.
		Thus,  $\Omega(L)$ has a basis of $Z$-monomials
		corresponding to a certain subset of $\textrm{Par}$.
	
		We give a brief comment on the proof of Theorem \ref{MT} in \S3.
		For $1\le{2i-1}\le{s}$, 
		$L(\Lambda^{(2i-1)}-(i-1)\delta)$ is included in
		$L(\Lambda_{0})\otimes{L(\Lambda_{1})}$ \cite[Theorem 4.6]{BM2}.
		On $L(\Lambda_{0})\otimes{L(\Lambda_{1})}$,  we have
		$Z_{\textrm{odd}}(\alpha_{1}+\cdots+\alpha_{j})=Z_{\textrm{even}}(\alpha_{s})=0$
		for $1\le{j}\le{s-1}$ (cf. \cite[\S3]{Kan}),
		and the vacuum spaces of level 2 standard modules 
		included in $L(\Lambda_{0})\otimes{L(\Lambda_{1})}$ 
		has a spanning set of $Z$-monomials consisting of 
		$Z_{\textrm{even}}(\alpha_{1}+\cdots+\alpha_{j})$'s and $Z_{\textrm{odd}}(\alpha_{s})$. 
		For $2\le{j}\le s-1$ and $v\in L(\Lambda_{0})\otimes{L(\Lambda_{1})}$,  
		$Z_{\textrm{even}}(\alpha_{1}+\cdots+\alpha_{j})$ acting on $v$ 
		is described as a span of $Z$-monomials consisting of
		$Z_{\textrm{even}}(\alpha_{1})$'s (Proposition \ref{beta_n}).
		Thus we 
		set
		$Z_{\textrm{even}}=Z_{\textrm{even}}(\alpha_{1})$ and
		$Z_{\textrm{odd}}=Z_{\textrm{odd}}(\alpha_{s})$,
		and give commutation relation between adjacent $Z$-operators
		using generalized commutation relation (Proposition \ref{pree},  \ref{preo},  \ref{proo}).
		Lepowsky-Wilson \cite{LW3} and Kanade \cite{Kan} arranged $Z$-monomials
		in the order of length and $T$ (Definition \ref{OrderT}).
		In this paper, 
		we add the order of numbers of odd parts in $Z$-monomials (cf. (\ref{order})).

	\subsubsection{Type $A^{(2)}_{5}$ and $A^{(2)}_{7}$ level 2}\label{Goll}
		Based on Theorem \ref{MT},
		we give $Z$-monomial generators of the vacuum spaces
		in the cases of type $A^{(2)}_{5}$ and $A^{(2)}_{7}$ level 2
		corresponding to the Göllnitz PT and the RRPT respectively
		(Theorem \ref{Thm5th} and Theorem \ref{Thm7th} below).
		The Göllnitz PT is stated as $G_{i}\stackrel{\rm{PT}}{\sim}T^{(8)}_{2i-1,4,9-2i}$ for $i=1,2$
		\cite[Chapter 7]{And2},
		where
		$G_{1}$ and $G_{2}$ are defined as follows,
		\begin{align*}
			G_{1}&=\{(\lambda_{1},\dotsc,\lambda_{l})\in R_{1}\mid
			l\ge0,
			\\*
			&\qquad\qquad\textrm{ if }
			\lambda_{i},\lambda_{i+1}\in2\mathbb{Z}
			\mbox{, then }{\lambda_{i}-\lambda_{i+1}\ge{4}}
			\mbox{ for }1\le i\le l-1
			\},\\
			G_{2}&=\{(\lambda_{1},\dotsc,\lambda_{l})\in G_{1}\mid
			\lambda_{l}\ge3\}.
		\end{align*}
		The principal characters $\chi(\Lambda^{(2i-1)})$
		in the case of $A^{(2)}_{5}$ (resp.  $A^{(2)}_{7}$), 
		which are equal to the infinite products of the GG (resp.  RR) identities, 
		are the generating functions for $T^{(8)}_{2i-1,4,9-2i}$ (resp.  $T^{(5)}_{i,5-i}$).
		We obtain the the $Z$-monomial bases of the vacuum spaces
		with indices satisfying the sum side partition conditions of the Göllnitz PT (resp.  RRPT).
		
		\begin{thm}\label{Thm5th}
			Consider the case of type $A^{(2)}_{5}$.
			Let $v_{\Lambda}$ be a highest weight vector of $L(\Lambda)$.
			Then 
			$\Omega(L(\Lambda^{(2i-1)}))$
			has the following basis for $i=1,2$,
			\begin{align*}
				\{Z_{i_{1}}\cdots{Z_{i_{l}}}v_{\Lambda^{(2i-1)}}
				\mid 
				l\ge0,\ (-i_{1},\dotsc,-i_{l})\in G_{i}
				\}.
			\end{align*}
		\end{thm}
		\begin{thm}\label{Thm7th}
			Consider the case of type $A^{(2)}_{7}$. 
			Let $v_{\Lambda}$ be a highest weight vector of $L(\Lambda)$.
			Then $\Omega(L(\Lambda^{(2i-1)}))$ 
			has the following basis for $i=1,2$,
			\begin{align*}
				\{Z_{i_{1}}\cdots{Z_{i_{l}}}v_{\Lambda^{(2i-1)}}
				\mid l\ge0,\ 
				(-i_{1},\dotsc,-i_{l})\in R_{i}
				\},
			\end{align*}
			where the sets $R_{i}$'s are defined in (\ref{RRS1}) and (\ref{RRS2}).
		\end{thm}
		Here,  we recall that
		$Z_{2n-1}=Z_{2n-1}(\alpha_{3})$ (resp.  $Z_{2n-1}(\alpha_{4})$)
		and $Z_{2n}=Z_{2n}(\alpha_{1})$
		for $n\in\mathbb{Z}$ in type $A^{(2)}_{5}$ (resp.  $A^{(2)}_{7}$).
		
		The proofs of Theorem \ref{Thm5th} and \ref{Thm7th}
		are essentially the same as
		the works by Lepowsky-Wilson \cite{LW3} and Kanade \cite{Kan},
		in the point that 
		we arrange $Z$-monomials
		in the order of length and $T$.
		Our results are different from the previous works (\cite{Kan},  \cite{BM}) in 
		that we take $Z_{\textrm{even}}(\alpha_{1})$ and $Z_{\textrm{odd}}(\alpha_{s})$ 
		as $Z$-operators ($s=3,4$).

	\subsubsection{Type $A^{(2)}_{9}$ level 2}
		For type $A^{(2)}_{9}$,  the product sides are as follows (See (\ref{Angelica})),
	\begin{align*}
		\chi(\Lambda^{(1)})
		&=
		\frac{1}{(q,q^{4},q^{6},q^{8},q^{11};q^{12})_{\infty}},\\
		\chi(\Lambda^{(3)})
		&=
		\frac{(q^{6};q^{12})_{\infty}}{(q^{2},q^{3},q^{4},q^{8},q^{9},q^{10};q^{12})_{\infty}},\\
		\chi(\Lambda^{(5)})
		&=
		\frac{1}{(q^{4},q^{5},q^{6},q^{7},q^{8};q^{12})_{\infty}}.
	\end{align*}
		Kanade-Russell \cite{KR} made conjectures of the sum side,
		which were proven by Bringmann et al \cite{BJM} and Rosengren \cite{Ros}.
		The conjectured sum side partition sets,
		$(K,) K_{1},K_{2}$ and $K_{3}$ are as follows \cite[\S3.1]{KR},
		\begin{align*}
				K&=\{(\lambda_{1},\dotsc,\lambda_{l})\in\textrm{Par}\mid
				l\ge0, \\*
				&\qquad(\textrm{k}1)\ \lambda_{i}-\lambda_{i+1}\neq1\textrm{ for }1\le i\le l-1,\\*
				&\qquad(\textrm{k}2)\ 
				\textrm{if }\lambda_{i}\in\mathbb{Z}\backslash2\mathbb{Z}, 
				\textrm{ then }\lambda_{i}\neq\lambda_{i+1}\textrm{ for }1\le i\le l-1,\\*
				&\qquad(\textrm{k}3)\ 
				\textrm{if } \lambda_{j+1}\in2\mathbb{Z}\textrm{ appears more than once, }\\*
				&
				\hspace{1.5cm}
				\textrm{then }
				\lambda_{j}-\lambda_{j+2}\ge4 \textrm{ for }2\le j\le l-2
				\},\\
				K_{1}&=\{(\lambda_{1},\dotsc,\lambda_{l})\in K\mid
				(\lambda_{l-1},  \lambda_{l})\neq(2,  2)\},\\
				K_{2}&=\{(\lambda_{1},\dotsc,\lambda_{l})\in K\mid\lambda_{l}\ge2\},\\
				K_{3}&=\{(\lambda_{1},\dotsc,\lambda_{l})\in K\mid\lambda_{l}\ge4\}.
		\end{align*}

		We reduce the $Z$-monomial spanning sets 
		to correspond to conjectures of integer partitions by Kanade-Russell.
		\begin{thm}\label{A29thm}
			Consider the case of type $A^{(2)}_{9}$.
			Let $v_{\Lambda}$ be a highest weight vector of $L(\Lambda)$.
			Then 
			$\Omega(L(\Lambda^{(2i-1)}))$ for $i=1,2,3$
			has the following basis,
			\begin{align*}
				\{Z_{i_{1}}\cdots{Z_{i_{l}}}v_{\Lambda}
				\mid l\ge0,\ (-i_{1},\dotsc,-i_{l})\in K_{i}
				\}.
			\end{align*}
			Here, 
			$Z_{2n-1}=Z_{2n-1}(\alpha_{5})$ and $Z_{2n}=Z_{2n}(\alpha_{1})$
			for $n\in\mathbb{Z}$.
		\end{thm}

		For the partition condition $(\textrm{k}3)$ of $K$,
		we rewrite the products,
		$Z_{\textrm{even}}(\alpha_{1}+\alpha_{2})Z_{\textrm{even}}(\alpha_{1})$,
		$Z_{\textrm{even}}(\alpha_{1})Z_{\textrm{even}}(\alpha_{1}+\alpha_{2})$,
		$Z_{\textrm{even}}(\alpha_{1}+\alpha_{2})Z_{\textrm{odd}}(\alpha_{s})$,
		and
		$Z_{\textrm{odd}}(\alpha_{s})Z_{\textrm{even}}(\alpha_{1}+\alpha_{2})$
		in the form of length 3 $Z$-monomials
		(see \S\ref{3contie} and \S\ref{conti_odd}).
		See also for example \cite{LW4} and \cite{T}
		for the relations among contiguous parts.

	\subsection{Future direction}
		For type $A^{(2)}_{11}$ level 2, 
		we obtain the same infinite products as in the case of type $A^{(2)}_{2}$ level 4. 
		Nandi \cite[pp.3-4]{Nan} made a conjecture of partition conditions
		for type $A^{(2)}_{2}$ level 4 
		as follows;
		for $(\lambda_{1},\lambda_{2},\dotsc,\lambda_{l})\in\textrm{Par}$,

		\begin{description}
			\item[N1]
				$\lambda_{i}-\lambda_{i+1}\neq1$ for $1\le i\le l-1$,
			\item[N2]
				$\lambda_{i}-\lambda_{i+2}\ge3$ for $1\le i\le l-2$,
			\item[N3]
				$\lambda_{i}-\lambda_{i+2}=3\Rightarrow\lambda_{i}\neq\lambda_{i+1}$
				for $1\le i\le l-2$,
			\item[N4]
				($\lambda_{i}-\lambda_{i+2}=3$ and $\lambda_{i}\notin2\mathbb{Z})		
				\Rightarrow\lambda_{i+1}\neq\lambda_{i+2}$ for $1\le i\le l-2$,
			\item[N5]
				($\lambda_{i}-\lambda_{i+2}=4$ and $\lambda_{i}\notin2\mathbb{Z})
				\Rightarrow(\lambda_{i}\neq\lambda_{i+1}$
				and $\lambda_{i+1}\neq\lambda_{i+2}$)
				for $1\le i\le l-2$,
			\item[N6]
				$(\lambda_{1}-\lambda_{2},  \lambda_{2}-\lambda_{3},\dotsc,
				\lambda_{l-1}-\lambda_{l})\neq(3,2^{*},3,0)$.
				\quad(Here,  $2^{*}$ is a sequence of non-negative number of 2.)
		\end{description}
		It is conceivable that the generating set in Theorem \ref{MT} can be reduced
		to satisfy these conditions by Nandi in the case of $A^{(2)}_{11}$. 
		In fact, to the spanning set in Theorem \ref{MT},  we can add a partition condition
		that the adjacent parts do not differ by 1, 
		which is included in the partition conditions of the subsets $G_{i}$'s,  $R_{i}$'s,  $K_{i}$'s, 
		and is just same as the condition N1 above.
		For detail,  see Remark \ref{rmk}.
		In addition,  Takigiku-Tsuchioka \cite{TT} made a conjecture of RR type identities
		in the case of $A^{(2)}_{13}$. 
		We expect that the RR type identities and partition theorems are found 
		in the cases of $A^{(2)}_{\textrm{odd}}$ with larger odd numbers.

	\subsection*{Organization}
		First,	in \S\ref{Prince},  following \cite[\S6]{Fi}, 
		we review the principal realization of type $A^{(2)}_{\textrm{odd}}$ affine Lie algebras
		(cf. \cite[Chapter 7-8]{Kac}). In \S\ref{Zop}, following \cite[\S7]{Fi} and \cite[\S3]{LW3}, 
		we review the definition of $Z$-operators. 
		Then in \S\ref{MTpf}, we prove Theorem \ref{MT}.
		In \S\ref{5th} and \S\ref{7th},
		we 
		prove Theorem \ref{Thm5th} and \ref{Thm7th}.
		Regarding the case of type $A^{(2)}_{9}$, 
		we prove Theorem \ref{A29thm} in \S\ref{9th}. 

	\subsection*{Acknowledgments}
		The author is grateful to Shunsuke Tsuchioka for his support and advice.
		This work is supported by RIKEN Junior Research Associate Program.

%

\section{Preliminaries}
%
%
	In this section, 
	we review the principal realization of
	the twisted affine Lie algebras of type $A^{(2)}_{\textrm{odd}}$
	and $Z$-operators following \cite{Fi} and \cite{KR}.

%
%

	\subsection{The principal realization of the affine Lie algebras
	$\hat{\mathfrak{g}}(\nu)$ and $\tilde{\mathfrak{g}}(\nu)$}\label{Prince}
		Let $s\ge3$ be a positive integer, and denote the root system of type $A_{2s-1}$ by $\Phi$. 
		Let $\Delta=\{\alpha_{1},\dotsc,\alpha_{2s-1}\}$ be the set of simple roots.
		Define $Q=\bigoplus^{2s-1}_{i=1}\mathbb{Z}\alpha_{i}$
		as the root lattice with the non-degenerate symmetric definite bilinear form
		$\langle\cdot,\cdot\rangle$ defined by
		\begin{align*}
			\langle\alpha_{i},\alpha_{j}\rangle
			=
			\begin{cases}
				2 & (\textrm{if}\hspace{2mm}i=j),\\
				-1 & (\textrm{if}\hspace{2mm}|i-j|=1),\\
				0 & (\textrm{otherwise}),\\
			\end{cases}
		\end{align*}
		for $1\le{i,j}\le2s-1$. 


		We denote by $\nu$ the twisted Coxeter automorphism (cf. \cite[\S8.2]{Fi})
		given by $\nu=\sigma_{1}\cdots\sigma_{s}\sigma$,
		where $\sigma$ is the Dynkin diagram involution 
		given by $\sigma(\alpha_{i})=\alpha_{2s-i}$ for $1\le i\le 2s-1$,
		and $\sigma_{i}$ is the reflection of $\alpha_{i}$ for $1\le i\le s$
		given by $\sigma_{i}(\alpha)=\alpha-\langle\alpha,\alpha_{i}\rangle\alpha_{i}$.
		The order of $\nu$ is $m=2(2s-1)$ (cf.  \cite[Proposition 8.2(2)]{Fi}). 
		The set of roots $\Phi$ breaks into $s$ $\nu$-orbits
		with a complete system of representatives
		$\Phi'=\{\beta_{i}\mid1\le{i}\le{s-1}\}\cup\{\alpha_{s}\}$
		where $\beta_{i}=\alpha_{1}+\alpha_{2}+\cdots+\alpha_{i}$. 
		The table below shows the $\nu$-orbits of $\alpha_{1}$ and $\alpha_{s}$.
		We omit $\nu^{i}\alpha_{1}$ and $\nu^{i}\alpha_{s}$ for $2s-1\le i\le 4s-3$,
		as $\nu^{2s-1}=-\textrm{id}$.
		This table is used for calculations in \S\ref{MTpf} and \S\ref{concrete579}. 
		\begin{table}[h]
			\caption{$\nu$-orbits of $\alpha_{1}$ and $\alpha_{s}$}
			\centering
			\begin{tabular}{|c||c|c|}\hline
				id & $\alpha_{1}$ & $\alpha_{s}$\\ 
				$\nu$ & $\alpha_{2s-1}$ & $-\alpha_{1}-\cdots-\alpha_{s}$\\ 
				$\nu^{2}$ & $\alpha_{2}$ & $-\alpha_{s+1}-\cdots-\alpha_{2s-1}$\\ 
				$\vdots$ & $\vdots$ & $\vdots$\\ 
				$\nu^{2k-1}$ & $\alpha_{2s-k}$ & $-\alpha_{k}-\cdots-\alpha_{s}$\\ 
				$\nu^{2k}$ & $\alpha_{k+1}$ & $-\alpha_{s+1}-\cdots-\alpha_{2s-k}$\\ 
				$\vdots$ & $\vdots$ & $\vdots$\\ 
				$\nu^{2s-4}$ & $\alpha_{s-1}$ & $-\alpha_{s+1}-\alpha_{s+2}$\\ 
				$\nu^{2s-3}$ & $\alpha_{1}+\cdots+\alpha_{s+1}$ & $-\alpha_{s-1}-\alpha_{s}$\\ 
				$\nu^{2s-2}$ & $\alpha_{s}+\cdots+\alpha_{2s-1}$ & $-\alpha_{s+1}$\\ \hline
			\end{tabular}
		\end{table}\\


		Let $\omega$ be a primitive $m$-th root of unity.
		As in \cite[\S6]{Fi}, define $\varepsilon:Q\times{Q}\rightarrow\mathbb{C}^{\times}$ by
		\begin{align*}
			\varepsilon(\alpha,\beta)
			=
			\prod_{p=1}^{m-1}(1-\omega^{-p})^{\langle\nu^{p}\alpha,\beta\rangle},
		\end{align*}
%
%
		and denote the complexification of the root lattice $Q$
		by $\mathfrak{a}$, i.e., $\mathfrak{a}=\mathbb{C}\otimes{Q}$. 
		Construct a finite dimensional simple Lie algebra $\mathfrak{g}$ of type $A_{2s-1}$
		using $\langle\cdot,\cdot\rangle$ and $\varepsilon$ by
		\begin{align*}
			\mathfrak{g}
			&
			=
			\mathfrak{a}\oplus\left(\bigoplus_{\alpha\in\Phi}\mathbb{C}x_{\alpha}\right)
			,\\
			{[\alpha_{i},x_{\alpha}]}
			=\langle\alpha_{i},\alpha\rangle{x_{\alpha}},&\qquad
			[x_{\alpha},x_{\beta}]
			=
			\begin{cases}
				\varepsilon(-\alpha,\alpha)\alpha 
				&(\textrm{if}\hspace{2mm}\langle\alpha,\beta\rangle=-2),\\	
				\varepsilon(\alpha,\beta)x_{\alpha+\beta} 
				&(\textrm{if}\hspace{2mm}\langle\alpha,\beta\rangle=-1),\\
				0 
				&(\textrm{otherwise}),
			\end{cases}
		\end{align*}
		where $x_{\alpha}$ is a symbol associated with $\alpha\in\Phi$ (cf.  \cite[p.21]{Fi}).
		We extend $\langle\cdot,\cdot\rangle$ and $\nu$ from $Q$ to $\mathfrak{g}$
		by (see \cite[p.22]{Fi})
		\begin{align*}
			\langle\alpha_{i},x_{\beta}\rangle
			=0,\qquad
			\langle{x_{\alpha}},x_{\beta}\rangle
			=\varepsilon(\alpha,\beta)\delta_{\alpha+\beta,0},\qquad
			\nu{x_{\alpha}}=x_{\nu\alpha}.
		\end{align*}


		For $i\in\mathbb{Z}$, 
		we define $\mathfrak{g}_{(i)}\subseteq\mathfrak{g}$
		as $\omega^{i}$-eigenspace of $\nu$, i.e.,
		\begin{align*}
			\mathfrak{g}_{(i)}
			=\{x\in\mathfrak{g}
			\mid\nu{x}=\omega^{i}x\}.
		\end{align*}

		Following \cite[pp.22-23]{Fi} (cf. \cite[Chapter 7-8]{Kac}), 
		construct the $\nu$-twisted affinizations
		$\hat{\mathfrak{g}}(\nu)$ and $\tilde{\mathfrak{g}}(\nu)$ by
		\begin{align*}
			&\hat{\mathfrak{g}}(\nu)
			=
			(\bigoplus_{i\in\mathbb{Z}}
			\mathfrak{g}_{(i)}\otimes{t^{i}})
			\oplus\mathbb{C}c,\qquad
			\tilde{\mathfrak{g}}(\nu)
			=
			\hat{\mathfrak{g}}(\nu)\oplus\mathbb{C}d,\\
			&{[x\otimes{t^{i}},y\otimes{t^{j}}]}
			=[x,y]\otimes{t^{i+j}}
			+\frac{1}{m}i\delta_{i+j,0}\langle{x},y\rangle{c},\\
			&[c,\tilde{\mathfrak{g}}(\nu)]=\{0\},\qquad
			[d,x\otimes{t^{i}}]=ix\otimes{t^{i}}.
		\end{align*}
		


		\begin{dfn}\cite[p.220]{LW3}
			Define the category $C_{k}$ as the category 
			consisting of $\tilde{\mathfrak{g}}(\nu)$-modules $V$,
			which satisfy
			\begin{enumerate}
				\item $c$ acts on $V$ as the scalar $k$,  i.e.,  $V$ has level $k$,
				\item $V=\coprod
				_{z\in\mathbb{C}}V_{z}$,
				\item For all $z\in\mathbb{C}$,  there exsists $i\in{\mathbb{Z}}_{\ge0}$, 
				such that $V_{z+j}=\{0\}$ for all $j>i$.
			\end{enumerate}
			Here,  $V_{z}=\{v\in{V}\mid{d\cdot{v}}=zv\}$.
		\end{dfn}


		The principal Heisenberg subalgebra
		$\hat{\mathfrak{a}}(\nu)\subseteq\tilde{\mathfrak{g}}(\nu)$
		is defined by (cf.  \cite[p.5]{Fi})
		\begin{align*}
			\hat{\mathfrak{a}}(\nu)
			=
			\bigoplus_{i\in\mathbb{Z}}(\mathfrak{a}\cap\mathfrak{g}_{(i)})
			\otimes{t^{i}}\oplus\mathbb{C}c.
		\end{align*}

		For $L\in{C_{k}}$,
		denote the vacuum space of $L$ by $\Omega(L)$, 
		which is the set of highest weight vectors of $\hat{\mathfrak{a}}(\nu)$, 
		i.e., $\Omega(L)=\{v\in{L}\mid\hat{\mathfrak{a}}_{+}(\nu)v=\{0\}\}$ \cite[p.28]{Fi}.
		If $L$ is standard,
		the principal character of $\Omega(L)$ is derived using Lepowsky numerator formula
		\cite[p.183]{LM}.

	\subsection{$Z$-operator}\label{Zop}

		Let $p_{i}:\mathfrak{g}\rightarrow\mathfrak{g}_{(i)}$ be
		the $i$-th projection. 
		For $x\in\mathfrak{g}$, 
		we define $x_{(i)}$ by $x_{(i)}=p_{i}(x)$.
		\begin{dfn}\label{Zdef} \cite[pp.221-222]{LW3}
			Let $L\in{C_{k}}$. 
			For $\alpha\in\Phi$, 
			define $X(\alpha,\zeta),
			E^{\pm}(\alpha,\zeta),
			Z(\alpha,\zeta)
			\in(\textrm{End}\ L)\{\zeta,\zeta^{-1}\}
			$ 
			as follows,
			\begin{align*}
				X(\alpha,\zeta)
				&=\sum_{n\in\mathbb{Z}}
				((x_{\alpha})_{(n)}\otimes{t^{n}})\zeta^{n},\\
				E^{\pm}(\alpha,\zeta)
				&=\exp\left(\pm{m}\sum_{n\ge1}(\alpha_{(\pm{n})}\otimes{t^{\pm{n}}})
				\zeta^{\pm{n}}/nk\right),\\
				Z(\alpha,\zeta)
				&=E^{-}(\alpha,\zeta)
				X(\alpha,\zeta)
				E^{+}(\alpha,\zeta).
			\end{align*}
			For $i\in\mathbb{Z}$, 
			we define $Z_{i}(\alpha)$ as the coefficient of $\zeta^{i}$ in $Z(\alpha, \zeta)$, 
			i.e., $Z(\alpha, \zeta)=\sum_{i\in\mathbb{Z}}Z_{i}(\alpha)\zeta^{i}$.
		\end{dfn}
		It is known that the vacuum space $\Omega(L)$ of 
		a highest weight module $L$ is generated by (cf. \cite[Theorem 7.1]{LW3})
		\begin{align}\label{genn}
			\{Z_{i_{1}}(\gamma_{1})\cdots{Z_{i_{l}}}(\gamma_{l})v_{L}\mid
			l\ge0,\ i_{j}\in\mathbb{Z}\mbox{ and }\gamma_{j}\in\Phi \mbox{ for }1\le i\le l\},
		\end{align}
		where $v_{L}$ is a highest weight of $L$.
		We have $Z(\nu^{p}\alpha,\zeta)=Z(\alpha,\omega^{p}\zeta)$ from \cite[Proposition 7.2]{Fi}, 
		thus we replace $\Phi$ to $\Phi'$,
		where $\Phi'$ is the complete system of representatives
		by $\nu$, $\Phi'=\{\beta_{i}\mid1\le{i}\le{s-1}\}\cup\{\alpha_{s}\}$.
		In addition,  from the highest weight property of $v_{L}$, 
		we can reduce the range of the indices from (\ref{genn}) to (cf.  \cite[\S 6]{Kan}),
		\begin{align}
			&\{Z_{i_{1}}(\gamma_{1})\cdots{Z_{i_{l}}}(\gamma_{l})v_{L}\mid
			l\ge0,\ i_{j}\in\mathbb{Z}\mbox{ and }
			\gamma_{j}\in\Phi'\mbox{ for }1\le i\le l,\notag\\*
			&\hspace{5cm}
			(i_{1},\dotsc,i_{l})\le_{T}(0,\dotsc,0)
			\},\label{reduced}
		\end{align}
		where the order $\le_{T}$ is defined as follows. 
		\begin{dfn}\label{OrderT}\cite[\S4]{Kan}
			Define the partial order $\le_{T}$ on $\mathbb{Z}^{l}$
			for any positive integer $l>0$,  as follows.
			For tuples $(i_{1},\dotsc,i_{l}), (j_{1},\dotsc,j_{l})\in\mathbb{Z}^{l}$,
			\begin{align*}
				(i_{1},\dotsc,i_{l})\le_{T}(j_{1},\dotsc,j_{l})
				\stackrel{\rm{def}}{\Leftrightarrow}
				i_{l'}+\cdots+i_{l}\le{j_{l'}+\cdots+j_{l}}\ 
				\textrm{for}\ 1\le{l'}\le{l}.
			\end{align*}
		\end{dfn}

		In \S3 and \S4,  we reduce the range of the indices and the roots to prove 
		Theorem 1.1 to 1.4.
%
%
		We use the following theorem to rewrite the adjacent $Z$-operators.
		\begin{thm} (Generalized commutation relations \cite[Theorem 3.10]{LW3})
			Let $L\in{C_{k}}$,  and
			$\zeta_{1}$ and $\zeta_{2}$ be commuting formal variables
			and let $\alpha, \beta\in\Phi$. Then,  we have
			\begin{align}
				&\prod_{p\in\mathbb{Z}/m\mathbb{Z}}
				(1-\omega^{-p}\zeta_{1}/\zeta_{2})^{\langle\nu^{p}\alpha,\beta\rangle/k}
				Z(\alpha,\zeta_{1})
				Z(\beta,\zeta_{2})\notag\\*
				&\qquad
				-\prod_{p\in\mathbb{Z}/m\mathbb{Z}}
				(1-\omega^{-p}\zeta_{2}/\zeta_{1})^{\langle\nu^{p}\beta,\alpha\rangle/k}
				Z(\beta,\zeta_{2})
				Z(\alpha,\zeta_{1})\notag\\*
				&=\frac{1}{m}
				\sum_{q\in{C_{-1}(\alpha,\beta)}}
				\varepsilon(\nu^{q}\alpha,\beta)
				Z(\nu^{q}\alpha+\beta,\zeta_{2})
				\delta(\omega^{-q}\zeta_{1}/\zeta_{2})\notag\\*
				&\qquad+\frac{k}{m^{2}}
				\langle{x_{\alpha},x_{-\alpha}}\rangle
				\sum_{q'\in{C_{-2}(\alpha,\beta)}}
				(D\delta)(\omega^{-q'}\zeta_{1}/\zeta_{2}).\label{thrm}
			\end{align}
			where, for $i=-1,-2$,
			\begin{equation*}
				C_{i}=
				\{p\in\mathbb{Z}/m\mathbb{Z}
				\mid\langle\nu^{p}\alpha,\beta\rangle=i\}.
			\end{equation*}
		\end{thm}

	%
	%

	\section{The case of $A^{(2)}_{\textrm{odd}}$ level 2}\label{MTpf}


		Recall that
		$\Lambda^{(n)}\coloneqq(\delta_{0,n}+\delta_{1,n})\Lambda_{0}+\Lambda_{n}$.
		From \cite[Theorem 4.6]{BM2},
		in the case of type $A^{(2)}_{2s-1}$, 
		we have
		for $1\le{2i-1}\le{s}$, 
			\begin{equation*}
				L(\Lambda^{(2i-1)}
				-(i-1)\delta)
				\subseteq{L(\Lambda_{0})\otimes{L(\Lambda_{1})}}.
			\end{equation*}
		On $L(\Lambda_{0})\otimes{L(\Lambda_{1})}$,  we have
			\begin{align}\label{tensor}
				Z_{2i+1}(\beta_{j})=Z_{2i}(\alpha_{s})=0
			\end{align}
		for $i\in\mathbb{Z}$ and $1\le{j}\le{s-1}$
		(See \cite[\S3]{Kan} \cite[p.279]{LW3} for the proof).
%
%
		As in the introduction,  define $Z_{i}$ for $i\in\mathbb{Z}$ by
		\begin{align*}
			Z_{i}
			=
			\begin{cases}
				Z_{i}(\alpha_{s}) & (i\in\mathbb{Z}\backslash2\mathbb{Z}),\\
				Z_{i}(\alpha_{1}) & (i\in2\mathbb{Z}).
			\end{cases}
		\end{align*} 
		
	\subsection{Auxiliary propositions}
		
		In this subsection, 
		let $v\in\Omega(L(\Lambda_{0})\otimes L(\Lambda_{1}))$.

		\begin{dfn}
			For $(i_{1},\dotsc,i_{l})\in\mathbb{Z}^{l}$ $(l\ge{1})$ and $n\ge0$, 
			define $\tilde T(i_{1},\dotsc,i_{l};v)$,  $T(i_{1},\dotsc,i_{l};v)$, $\tilde{S}(n;v)$
			and ${S}(n;i_{1},\cdots,i_{l};v)$ as follows,
			\begin{align}
				\tilde{T}(i_{1},\dotsc,i_{l};v)&=\{Z_{j_{1}}\cdots Z_{j_{l'}}v
				\mid l'<l,\ \mbox{or}\ (l=l'\ \mbox{and}\ (i_{1},\dotsc,i_{l})<_{T}(j_{1},\dotsc,j_{l'}))\},
				\notag\\
				T(i_{1},\dotsc,i_{l};v)&=
				\SPAN(\tilde T(i_{1},\dotsc,i_{l};v)),\label{OT}\\
				\tilde{S}(n;v)&=\{Z_{i_{1}}\cdots Z_{i_{l'}}v
				\mid l'\ge0,\ |\{i_{j}\in\mathbb{Z}\backslash2\mathbb{Z}
				\mid1\le j\le l'\}|=n\},\label{order}\\
				S(n;i_{1},\cdots,i_{l};v)
				&=\SPAN(\tilde{S}(n;v)\cap{\tilde{T}(i_{1},\dotsc,i_{l};v)}).\notag
			\end{align}
		\end{dfn}

		Let 
		$F(\alpha,\beta,\zeta)
		=\prod_{t\in\mathbb{Z}/m\mathbb{Z}}
		(1-\omega^{-t}\zeta)^{\langle\nu^{t}\alpha,\beta\rangle/2}$,
		and $c(\alpha,\beta,n)$ be the coefficient of $\zeta^{n}$ in $F(\alpha,\beta,\zeta)$.
		By comparing the coefficients of $\zeta_{1}^{a}\zeta_{2}^{b}$ for $a,b\in\mathbb{Z}$
		in (\ref{thrm}), 
		we get
		\begin{align*}
			&\sum_{p\ge0}c(\alpha,\beta,p)Z_{a-p}(\alpha)Z_{b+p}(\beta)-
			c(\beta,\alpha,p)Z_{b-p}(\beta)Z_{a+p}(\alpha)\notag\\*
			&=\frac{1}{m}\sum_{q\in{C_{-1}(\alpha,\beta)}}\varepsilon(\nu^{q}\alpha,\beta)						\omega^{-2qa}Z_{a+b}(\nu^{q}\alpha+\beta)+d(\alpha,\beta,a,b), 
		\end{align*}
		where $d(\alpha,\beta,a,b)$ is a constant determined by $\alpha,\beta,a,$ and $b$.
		

		\begin{prop}\label{pree}
			For $i,i'\in2\mathbb{Z}$ such that $i>i'$,  we have
			\begin{enumerate}[(a)]
				\item$Z_{i}Z_{i'}v
					\in{S}(0;i-1,i'+1;v)$,
				\item$Z_{i}Z_{i'}v
					\in{S}(0;i,i';v)$,
				\item$Z_{i+i'}(\beta_{2})v
					\in{S}(0;i-1,i'+1;v)$.
			\end{enumerate}
		\end{prop}


		\begin{proof}
			It is enough to show (a) and (c),
			as we have
			${S}(0;i-1,i'+1;v)\subseteq{S}(0;i,i';v)$
			from the definitions of ${S}$.
			Take $i,i'\in2\mathbb{Z}$ such that $i>i'$. 
			We denote ${S}(0;i-1,i'+1;v)$
			by ${S}(i,i')$.
			As $C_{-1}(\alpha_{1},\alpha_{1})=\{2,m-2\}$,
			by assigning $\alpha=\beta=\alpha_{1}$ and $k=2$ in (\ref{thrm}), 
			and comparing the coefficients of $\zeta_{1}^{i}\zeta_{2}^{i'}$ and
			$\zeta_{1}^{i+1}\zeta_{2}^{i'-1}$ respectively,
			we get the following formulae,
			\begin{align}
				&(c(\alpha_{1},\alpha_{1},0)Z_{i}Z_{i'}
				-g_{1}(i,i')Z_{i+i'}(\beta_{2}))v
				\in
				{{S}(i,i')},\label{ij1}\\
				&(c(\alpha_{1},\alpha_{1},1)Z_{i}Z_{i'}
				-g_{1}(i+1,i'-1)Z_{i+i'}(\beta_{2}))v
				\in
				{S}(i,i'),\label{ij2}
			\end{align}
			where for $n,n'\in\mathbb{Z}$,
			\begin{align*}
				g_{1}(n,n')=\frac{1}{m}\varepsilon(\nu^{2}\alpha_{1},\alpha_{1})
				(\omega^{-2n}-\omega^{-2n'}).
			\end{align*}

			It is sufficient to show that
			the following determinant of the matrix, 
			consisting of the coefficients of $Z_{i}Z_{i'}$
			and $Z_{i+i'}(\beta_{2})$ in (\ref{ij1}) and (\ref{ij2}),  is not 0 
			in the cases of $i-i'\equiv2,4\textrm{ (mod }{4s-2)}$.
			\begin{align*}
				\begin{vmatrix} 
					c(\alpha_{1},\alpha_{1},0) & c(\alpha_{1},\alpha_{1},1)\\
					g_{1}(i,i') & g_{1}(i+1,i'-1)
				\end{vmatrix}
				&=\frac{1}{m}\varepsilon(\nu^{2}\alpha_{1},\alpha_{1})
				\omega^{-2i}((2-\omega^{2})-\omega^{2i-2i'-2}(2\omega^{2}-1)).
			\end{align*}
			The reason we consider only the cases of $i-i'\equiv2,4$ is
			as follows.
			If the given determinant is not 0 in such cases, 
			we have $Z_{i}Z_{i'}v\in{{S}(i,i')}$ and $Z_{i+i'}(\beta_{2})v\in{{S}(i,i')}$
			for $i,i'\in2\mathbb{Z}$ such that $i>i'$ and $i-i'\equiv2,4$.
			The latter one means, for all $i,i'\in2\mathbb{Z}$ such that $i>i'$, 
			we have
			$Z_{i+i'}(\beta_{2})v\in{{S}(i,i')}$.
			Thus, from (\ref{ij1}),  we have $Z_{i}Z_{i'}v\in{{S}(i,i')}$
			for all $i,i'\in2\mathbb{Z}$ such that $i>i'$.

			Set $f(n)=(2-\omega^{2})-\omega^{2n-2}(2\omega^{2}-1)$. Then, we have
			\begin{align*}
				f(2)=2-2\omega^{4}=-2(\omega^{4}-1),\qquad
				f(4)=(\omega^{4}-1)(-2\omega^{4}+\omega^{2}-2).
			\end{align*}
			It is easy to see $f(2),f(4)\neq0$. 
			Therefore,  we obtain the consequence.
		\end{proof}


		\begin{prop}\label{preo}
			For
			$i\in2\mathbb{Z}$ and $j\in\mathbb{Z}\backslash2\mathbb{Z}$
			such that $i>j$,  we have
			\begin{align*}
				Z_{i}Z_{j}v
				\in{S}(1;i,j;v).
			\end{align*}
			Likewise, for 
			$j\in\mathbb{Z}\backslash2\mathbb{Z}$ and $i\in2\mathbb{Z}$ such that $j>i$, 
			we have
			\begin{align*}
				Z_{j}Z_{i}v
				\in{S}(1;j,i;v).
			\end{align*}
		\end{prop}


		\begin{proof}
			Take 
			$i\in2\mathbb{Z}$ and $j\in\mathbb{Z}\backslash2\mathbb{Z}$
			such that $i>j$. 
			As $C_{-1}(\alpha_{1},\alpha_{s})=\{2s-4,m-1\}$,
			by assigning $\alpha=\alpha_{1}$, $\beta=\alpha_{s}$, and $k=2$
			in (\ref{thrm}), and comparing the coefficients of $\zeta_{1}^{i}\zeta_{2}^{j}$,
			we obtain
			\begin{align*}
				c(\alpha_{1},\alpha_{s},0)Z_{i}Z_{j}v
				=Z_{i}Z_{j}v
				\in{{S}(1;i,j;v)}.
			\end{align*}

			Likewise,  take 
			$j\in\mathbb{Z}\backslash2\mathbb{Z}$ and $i\in2\mathbb{Z}$
			such that $j>i$.
			By assigning $\alpha=\alpha_{1}$, $\beta=\alpha_{s}$, and $k=2$ in (\ref{thrm}),
			and comparing the coefficients of $\zeta_{1}^{j}\zeta_{2}^{i}$,
			we get
			\begin{align*}
				c(\alpha_{s},\alpha_{1},0)Z_{j}Z_{i}v
				=Z_{j}Z_{i}v\in{{S}(1;j,i;v)}.
			\end{align*}
		\end{proof}


		\begin{prop}\label{beta_n}
			For $a\in2\mathbb{Z}$ and $2\le n\le s-1$,  we have
			\begin{align*}
				Z_{a}(\beta_{n})v\in\SPAN(\tilde{S}(0;v)).
			\end{align*}
		\end{prop}

		\begin{proof}
			We use induction on $n$,  and denote $\SPAN(\tilde{S}(0;v))$ by $S'$.
			From Proposition \ref{pree}(c),  we have $Z_{a}(\beta_{2})v\in{S'}$.
			Set $2\le n\le s-2$ and take $i,i'\in2\mathbb{Z}$ such that $i+i'=a$ and $i\neq{i'}$.
			Assume that
			$Z_{\textrm{even}}(\beta_{n-1})v,
			Z_{\textrm{even}}(\beta_{n})v\in{S'}$.
			As $C_{-1}(\alpha_{1},\beta_{n})=\{2i,2s-1,2s+2i-3,m-2\}$,
			by assigning $\alpha=\alpha_{1}$,  $\beta=\beta_{n}$,  and $k=2$ in (\ref{thrm}), 
			and comparing the coefficients of $\zeta_{1}^{i}\zeta_{2}^{i'}$, 
			we obtain
			\begin{align}
				&
				g_{2}(i,i',n)
				Z_{i+i'}(\beta_{n+1})v\notag\\*
				&=(\sum_{p=0}^{\infty}(c(\alpha_{1},\beta_{n},p)
				Z_{i-p}(\alpha_{1})Z_{i'+p}(\beta_{n})
				-c(\beta_{n},\alpha_{1},p)Z_{i'-p}(\beta_{n})Z_{i+p}(\alpha_{1}))\notag\\*
				&
				-g_{3}(i,i',n)
				Z_{i+i'}(\beta_{n-1}))v,\label{beta_n+}
			\end{align}
			where
			\begin{align*}
				g_{2}(i,i',n)
				&=
				\varepsilon(\nu^{2n}\alpha_{1},\beta_{n})\omega^{-2ni}
				+\varepsilon(\nu^{-2}\alpha_{1},\beta_{n})\omega^{-2i'}\\*
				&=
				\varepsilon(\nu^{2n}\alpha_{1},\beta_{n})
				(\omega^{-2ni}-\omega^{-2i'}),\\
				g_{3}(i,i',n)
				&=\varepsilon(-\alpha_{1},\beta_{n})\omega^{2(i+i')}
				+\varepsilon(-\nu^{2n-2}\alpha_{1},\beta_{n})\omega^{(-2n+2)i}.
			\end{align*}
			By taking $i,i'$ properly,  we obtain $\omega^{-2ni}-\omega^{-2i'}\neq0$, 
			i.e., $g_{2}(i,i',n)\neq0$.
			Thus from (\ref{beta_n+}), 
			we get $Z_{\textrm{even}}(\beta_{n+1})v\in{S'}$. 
			Therefore we obtain the consequence.
		\end{proof}

		\begin{prop}\label{proo}
			For $j,j'\in\mathbb{Z}\backslash2\mathbb{Z}$ such that $j>j'$, we have
			\begin{align*}
				Z_{j}Z_{j'}v\in{\SPAN((\tilde{S}(2;v)\cap\tilde{T}(j,j';v))\cup\tilde{S}(0;v))}.
			\end{align*}
		\end{prop}

		\begin{proof}
			Take $j,j'\in\mathbb{Z}\backslash2\mathbb{Z}$ such that $j>j'$. 
			As $C_{-1}(\alpha_{s},\alpha_{s})
			=\{n\in\mathbb{Z}\backslash2\mathbb{Z}\mid1\le{n}\le{m-1},\ n\neq2s-1\}$,
			by assigning $\alpha=\beta=\alpha_{s}$, $k=2$ in (\ref{thrm}), 
			and comparing the coefficients of $\zeta_{1}^{j}\zeta_{2}^{j'}$, we obtain
			\begin{align*}
				&(c(\alpha_{s},\alpha_{s},0)Z_{j}Z_{j'}
				-\frac{1}{m}\sum^{s-1}_{q=1}
				g_{4}(j,j',q)
				Z_{j+j'}(\beta_{s-q}))v\\*
				&=(Z_{j}Z_{j'}
				-\frac{1}{m}\sum^{s-1}_{q=1}
				g_{4}(j,j',q)
				Z_{j+j'}(\beta_{s-q}))v
				\in{S}(2;j,j';v),
			\end{align*}
			where $g_{4}(j,j',q)$ is a constant dependent on $j,j'$ and $q$.
			We get
			\begin{align*}
				(\frac{1}{m}\sum^{s-1}_{q=1}
				g_{4}(j,j',q)Z_{j+j'}(\beta_{s-q}))v\in\SPAN(\tilde{S}(0;v)),
			\end{align*}
			from Proposition \ref{beta_n}.  Thus we get the consequence.
		\end{proof}

		\subsection{Proof of the theorem}
			We reduce the $Z$-monomial generating set of the vacuum space
			from (\ref{reduced}) to (\ref{MT'}).
			From (\ref{tensor}) and Proposition \ref{beta_n}, 
			$Z_{\textrm{even}}(\alpha_{s})$, 
			$Z_{\textrm{odd}}(\alpha_{1})$ and
			$Z_{i}(\beta_{n})$ for $i\in\mathbb{Z}$ are eliminated from (\ref{reduced}),
			remaining
			\begin{align}\label{remaining}
				&\{Z_{i_{1}}\cdots{Z_{i_{l}}}v_{L}\mid
				l\ge0,\ i_{j}\in\mathbb{Z}\mbox{ for }1\le i\le l,\ 
				(i_{1},\dotsc,i_{l})\le_{T}(0,\dotsc,0)
				\}.
			\end{align}
			Thus it is sufficient to show that 
			$Z$-monomials $Z_{i_{1}}\cdots{Z_{i_{l}}}v_{L}$ with some parts $(i_{j},i_{j+1})$
			such that $i_{j}>i_{j+1}$ are eliminated from (\ref{remaining}).
			Take such $Z_{i_{1}}\cdots{Z_{i_{l}}}v_{L}$
			and
			assume that $Z_{i_{1}}\cdots Z_{i_{l}}v_{L}\in S(n;v_{L})$.
			From Proposition \ref{proo},
			if some parts $(i_{j},i_{j+1})$
			such that $i_{j}>i_{j+1}$
			are odd-odd parts,
			we have
			\begin{align*}
			Z_{i_{1}}\cdots Z_{i_{l}}v_{L}\in
			{\SPAN((\tilde{S}(n;v)\cap\tilde{T}(i_{1},\dotsc,i_{l};v))\cup\tilde{S}(n-2;v))}.
			\end{align*}
			From Proposition \ref{pree}(b) and \ref{preo},
			if no part $(i_{j},i_{j+1})$
			such that $i_{j}>i_{j+1}$
			is an odd-odd part,
			then we have
			\begin{align*}
			Z_{i_{1}}\cdots Z_{i_{l}}v_{L}\in{\tilde{S}(n;i_{1},\dotsc,i_{l};v_{L})}.
			\end{align*}
			Therefore,  $Z_{i_{1}}\cdots{Z_{i_{l}}}v_{L}$
			are eliminated from (\ref{remaining}).

		\begin{rmk}\label{rmk}
			In Theorem \ref{MT}, we can add the partition condition
			\begin{align*}
				i_{n+1}-i_{n}\neq1,
			\end{align*}
			for $1\le{n}\le{l}$,
			since we have
			\begin{enumerate}
			\item
			$Z_{i}Z_{i+1}v\in{S}(1;i,i+1;v)$ for $i\in2\mathbb{Z}$,
			\item
			$Z_{j}Z_{j+1}v\in{S}(1;j,j+1;v)$ for $j\in\mathbb{Z}\backslash2\mathbb{Z}$,
			\end{enumerate}
			We show (1) and (2) briefly.
			By assigning $\alpha=\alpha_{1}$, $\beta=\alpha_{s}$, and $k=2$ in (\ref{thrm}),	
			and comparing the coefficients of $\zeta_{1}^{i+1}\zeta_{2}^{i}$ for $i\in2\mathbb{Z}$,
			we obtain 
			\begin{align}\label{rmk1}
				c(\alpha_{1},\alpha_{s},1)Z_{i}Z_{i+1}v
				=Z_{i}Z_{i+1}v
				\in{S}(1;i,i+1;v),
			\end{align} 
			and by comparing the coefficients of $\zeta_{1}^{j+1}\zeta_{2}^{j}$
			for $j\in\mathbb{Z}\backslash2\mathbb{Z}$, 
			we have
			\begin{align}\label{rmk2}
				c(\alpha_{s},\alpha_{1},1)Z_{j}Z_{j+1}v
				=Z_{j}Z_{j+1}v
				\in{S}(1;j,j+1;v).
			\end{align}

		\end{rmk}

\section{The cases of $A^{(2)}_{5}$, $A^{(2)}_{7}$, and $A^{(2)}_{9}$ type level 2}\label{concrete579}
		In this section,  we consider the cases of 
		$A^{(2)}_{2s-1}$ type level 2 for $s=3,4,5$.
		In these cases,  we prove $Z_{i_{1}}\cdots Z_{i_{l}}v\in{T(i_{1},\dotsc,i_{l};v)}$
		for $(i_{1},\dotsc,i_{l})\in\mathbb{Z}^{l}$
		which does not satisfy certain partition conditions
		(See (\ref{OT}) for the definition of $T(i_{1},\dotsc,i_{l};v)$).
		In Proposition \ref{pree}(b),  \ref{preo},  and (\ref{rmk1}),  (\ref{rmk2}),
		we have proven that $Z_{i}Z_{j}v\in{T(i,j;v)}$
		in the cases such that
		\begin{enumerate}
			\item $i,j\in2\mathbb{Z}$ and $i>j$,
			\item $i\in2\mathbb{Z}$, $j\in\mathbb{Z}\backslash2\mathbb{Z}$ and $i>j$,
			\item $i\in\mathbb{Z}\backslash2\mathbb{Z}$, $j\in2\mathbb{Z}$ and $i>j$,
			\item $j-i=1$. 
		\end{enumerate}
		In addition, from Proposition \ref{beta_n}, 
		we have for $1\le{n}\le{s-1}$,
		\begin{align*}
			Z_{\textrm{even}}(\beta_{n})v
			\in\SPAN\{Z_{i_{1}}Z_{i_{2}}\cdots{Z_{i_{l}}}v\mid l\ge0\}.
		\end{align*}
		Here,  we focus on the remaining partition conditions.
		As in the previous section,  
		let $v\in\Omega(L)$ for $L=L(\Lambda_{0})\otimes{L(\Lambda_{1})}$. 

	\subsection{The case of type $A^{(2)}_{5}$}\label{5th}
		We prove Theorem \ref{Thm5th} in this section.
		The remaining conditions are as follows.
		\begin{enumerate}
			\item
				$Z_{j}Z_{j'}v\in{{T}(j,j';v)}$
				for $j,j'\in\mathbb{Z}\backslash2\mathbb{Z}$ such that $j\ge{j'}$.
			\item
				$Z_{i}Z_{i}v\in{{T}(i,i;v)}$ for $i\in2\mathbb{Z}$.
			\item
				$Z_{i}Z_{i+2}v\in{{T}(i,i+2;v)}$ for $i\in2\mathbb{Z}$.
			\item
				${i_{l}}\le{-3}$ in the case of $L(\Lambda_{3})$.
		\end{enumerate}

		For $m,n\in\mathbb{Z}$, define $h_{1}(m,n)$ and $h_{2}(m,n)$ as follows.
		\begin{align*}
			h_{1}(m,n)=\frac{1}{10}\varepsilon(\nu\alpha_{3},\alpha_{3})
			(\omega^{-m}-\omega^{-n}),\quad
			h_{2}(m,n)=\frac{1}{10}\varepsilon(\nu^{2}\alpha_{1},\alpha_{1})
			(\omega^{-m}-\omega^{-n}).
		\end{align*}

		\begin{prop}
			For $j,j'\in\mathbb{Z}\backslash2\mathbb{Z}$ such that $j>j'$,  we have
			$Z_{j}Z_{j'}v\in{{T}(j,j';v)}$.
		\end{prop}
		\begin{proof}
			By assigning $\alpha=\beta=\alpha_{3}$ and $k=2$ in (\ref{thrm}), 
			and comparing the coefficients of $\zeta_{1}^{j}\zeta_{2}^{j'}$, 
			we get
			\begin{align*}
				&(c(\alpha_{3},\alpha_{3},0)Z_{j}Z_{j'}
				-h_{1}(j,j')
				Z_{j+j'}(\beta_{2}))v\\*
				&=
				(Z_{j}Z_{j'}
				-h_{1}(j,j')
				Z_{j+j'}(\beta_{2}))v
				\in{T(j,j';v)}.
			\end{align*}
			From Proposition \ref{pree}, we get 
			\begin{align*}
				Z_{j+j'}(\beta_{2})v=Z_{(j+1)+(j'-1)}(\beta_{2})v\in{T((j+1)-1,(j'-1)+1;v)}=T(j,j';v),
			\end{align*}
			thus $Z_{j}Z_{j'}v\in{T(j,j';v)}$ holds.
		\end{proof}

		\begin{prop}
			For $j\in\mathbb{Z}\backslash2\mathbb{Z}$,
			we have
			\begin{enumerate}[(a)]
				\item$Z_{j}Z_{j}v\in{T(j,j;v)}$,
				\item$Z_{2j}(\beta_{2})v\in{T(j-1,j+1;v)}$.
			\end{enumerate}
			((b) is necessary for proving Proposition \ref{25+2}.)
		\end{prop}

		\begin{proof}
			By assigning $\alpha=\beta=\alpha_{3}$ and $k=2$ in (\ref{thrm}), 
			and comparing the coefficients of
			$\zeta_{1}^{j+1}\zeta_{2}^{j-1}$, 
			$\zeta_{1}^{j+2}\zeta_{2}^{j-2}$ and
			$\zeta_{1}^{j+3}\zeta_{2}^{j-3}$ respectively,
			we get
			\begin{align}
				&(c(\alpha_{3},\alpha_{3},1)Z_{j}Z_{j}
				-h_{1}(j+1,j-1)
				Z_{2j}(\beta_{2}))v
				\in{T(j-1,j+1;v)},\label{5-1}\\
				&(c(\alpha_{3},\alpha_{3},0)Z_{j+2}Z_{j-2}
				+c(\alpha_{3},\alpha_{3},2)Z_{j}Z_{j}
				\notag\\*
				&
				-h_{1}(j+2,j-2)Z_{2j}(\beta_{2}))v
				\in{T(j-1,j+1;v)},\label{5-2}\\
				&(c(\alpha_{3},\alpha_{3},1)Z_{j+2}Z_{j-2}
				+c(\alpha_{3},\alpha_{3},3)Z_{j}Z_{j}\notag\\*
				&-h_{1}(j+3,j-3)Z_{2j}(\beta_{2}))v
				\in{T(j-1,j+1;v)}.\label{5-3}
			\end{align}

			The determinant of the matrix consisting of the coefficients of
			$Z_{j+2}Z_{j-2}$, 
			$Z_{j}Z_{j}$
			and $Z_{2j}(\beta_{2})$ in (\ref{5-1}), (\ref{5-2}), and (\ref{5-3}) is,
			\begin{align*}
				&
				\begin{vmatrix}
					0 & c(\alpha_{3},\alpha_{3},0) & c(\alpha_{3},\alpha_{3},1)\\
					c(\alpha_{3},\alpha_{3},1)
					& c(\alpha_{3},\alpha_{3},2)
					& c(\alpha_{3},\alpha_{3},3)\\
					h_{1}(j+1,j-1)
					& h_{1}(j+2,j-2)
					& h_{1}(j+3,j-3)\\
				\end{vmatrix}\\
				&=h_{1}(j+1,j-1)(-2\omega^3 + 2\omega^2 + 2)\neq0.
			\end{align*}
			Thus, we obtain
			\begin{align}\label{beta_2}
				Z_{j}Z_{j}v, Z_{2j}(\beta_{2})v
				\in{T(j-1,j+1;v)}(\subseteq{T(j,j;v)}).
			\end{align}
		\end{proof}

		\begin{prop}\label{25+2}
			We have $Z_{i}Z_{i+2}v\in{{T}(i,i+2;v)}$
			for $i\in2\mathbb{Z}$.
		\end{prop}

		\begin{proof}
			By assigning $\alpha=\beta=\alpha_{1}$ and $k=2$ in (\ref{thrm}),
			and comparing the coefficients of
			$\zeta_{1}^{i+2}\zeta_{2}^{i}$ and $\zeta_{1}^{i+3}\zeta_{2}^{i-1}$ respectively,
			we get
			\begin{align}
				&(c(\alpha_{1},\alpha_{1},0)Z_{i+2}Z_{i}
				+(c(\alpha_{1},\alpha_{1},2)
				-c(\alpha_{1},\alpha_{1},0))Z_{i}Z_{i+2}\notag\\*
				&
				-h_{2}(2i+4,2i)
				Z_{2i+2}(\beta_{2}))v
				\in{T(i,i+2;v)},\label{cat}\\
				&(c(\alpha_{1},\alpha_{1},1)Z_{i+2}Z_{i}
				+c(\alpha_{1},\alpha_{1},3)Z_{i}Z_{i+2}\notag\\*
				&-h_{2}(2i+6,2i-2)
				Z_{2i+2}(\beta_{2}))v
				\in{T(i,i+2;v)}.\label{cat'}
			\end{align}
			From (\ref{beta_2}), we have
			\begin{align*}
				Z_{2i+2}(\beta_{2})v=Z_{2(i+1)}(\beta_{2})v\in{T((i+1)-1,(i+1)+1;v)}=T(i,i+2;v),
			\end{align*}
			thus,  we have to calculate the determinant of the matrix
			consisting of the coefficients of
			$Z_{i+2}(\alpha_{1})Z_{i}(\alpha_{1})$ and $Z_{i}(\alpha_{1})Z_{i+2}(\alpha_{1})$
			in (\ref{cat}) and (\ref{cat'}).
			\begin{align*}
				\begin{vmatrix}
					c(\alpha_{1},\alpha_{1},0)
					& c(\alpha_{1},\alpha_{1},1)\\
					c(\alpha_{1},\alpha_{1},2)-c(\alpha_{1},\alpha_{1},0)
					& c(\alpha_{1},\alpha_{1},3)
				\end{vmatrix}
				=6\omega^3-6\omega^2+2\neq0.
			\end{align*}
			Therefore, we obtain
			$Z_{i}Z_{i+2}v\in{T(i,i+2;v)}$ for $i\in2\mathbb{Z}$.
		\end{proof}

		\begin{prop}
			We have $Z_{i}Z_{i}v\in{{T}(i,i;v)}$
			for $i\in2\mathbb{Z}$. 	
		\end{prop}
	
		\begin{proof}
			First we show $Z_{2j+2}(\beta_{2})v\in{T(j+1,j+1;v)}$
			for $j\in\mathbb{Z}\backslash2\mathbb{Z}$.
			By assigning $\alpha=\beta=\alpha_{3}$ and $k=2$ in (\ref{thrm}), 
			and comparing the coefficients of
			$\zeta_{1}^{j+2}\zeta_{2}^{j}$ and $\zeta_{1}^{j+3}\zeta_{2}^{j-1}$ respectively,
			we get
			\begin{align}
				&(c(\alpha_{3},\alpha_{3},0)Z_{j+2}Z_{j}
				-h_{1}(j+2,j)
				Z_{2j+2}(\beta_{2}))v
				\in{T(j+1,j+1;v)},\label{5_4}\\
				&(c(\alpha_{3},\alpha_{3},1)Z_{j+2}Z_{j}
				-h_{1}(j+3,j-1)
				Z_{2j+2}(\beta_{2}))v
				\in{T(j+1,j+1;v)}.\label{5_5}
			\end{align}
			The determinant of the matrix consisting of the coefficients of
			$Z_{j+2}(\alpha_{3})Z_{j}(\alpha_{3})$ and $Z_{2j+2}(\beta_{2})$
			in (\ref{5_4}) and (\ref{5_5}) is,
			\begin{align*}
				\begin{vmatrix}
					c(\alpha_{3},\alpha_{3},0) & c(\alpha_{3},\alpha_{3},1)\\
					h_{1}(j+2,j-2) & h_{1}(j+3,j-3)
				\end{vmatrix}
				&=h_{1}(j+2,j-2)(\omega^{-1}+\omega+1)\neq0.
			\end{align*}
			Thus,  for $j\in\mathbb{Z}\backslash2\mathbb{Z}$ we obtain 
			\begin{align}\label{dog}
				Z_{2j+2}(\beta_{2})v\in{T(j+1,j+1;v)},
			\end{align}
			and now,  we show $Z_{i}Z_{i}v\in{{T}(i,i;v)}$
			for $i\in2\mathbb{Z}$. 
			By assigning $\alpha=\beta=\alpha_{1}$ and $k=2$ in (\ref{thrm}), 
			and comparing the coefficients of $\zeta_{1}^{i+1}\zeta_{2}^{i-1}$, 
			we get
			\begin{align*}
				(c(\alpha_{1},\alpha_{1},1)Z_{i}Z_{i}
				-h_{2}(2i+2,2i-2)
				Z_{2i}(\beta_{2}))v
				\in{T(i,i;v)}.
			\end{align*}
			We have $c(\alpha_{1},\alpha_{1},1)\neq0$, and from (\ref{dog}), we have 
			\begin{align*}
				Z_{2i}(\beta_{2})v=Z_{2(i-1)+2}(\beta_{2})v\in{T((i-1)+1,(i-1)+1;v)}=T(i,i;v).
			\end{align*}
			Thus
			we obtain $Z_{i}Z_{i}v\in{T(i,i;v)}$ for $i\in2\mathbb{Z}$.
		\end{proof}

			The condition $i_{l}\le{-3}$ in the case of $L(\Lambda_{3})$
			follows from the principal character $\chi(L(\Lambda_{3}))$.
			By the G\"ollnitz-Gordon identities,  we see that the given sets are bases.

	\subsection{The case of type $A^{(2)}_{7}$}\label{7th}
		We prove Theorem \ref{Thm7th}. 
			In this case, the remaining conditions are as follows.
			\begin{enumerate}
				\item
					$Z_{j}Z_{j'}v\in{{T}(j,j';v)}$
					for $j,j'\in\mathbb{Z}\backslash2\mathbb{Z}$ such that $j\ge j'$.
				\item
					$Z_{i}Z_{i}v\in{{T}(i,i;v)}$ for $i\in2\mathbb{Z}$.
				\item
					$i_{l}\le-2$ in the case of $L(\Lambda_{3})$.
			\end{enumerate}
		For $l,m,n\in\mathbb{Z}$, define $h_{1}(l,m,n)$ and $h_{2}(l,m,n)$ as follows.
		\begin{align*}
			h_{1}(l,m,n)&=\frac{1}{14}\varepsilon(\nu^{l}\alpha_{4},\alpha_{4})
			(\omega^{-m}-\omega^{-n}),\\
			h_{2}(l,m,n)&=\frac{1}{14}\varepsilon(\nu^{l}\alpha_{1},\alpha_{1})
			(\omega^{-m}-\omega^{-n}).
		\end{align*}
		\begin{prop}
			We have $Z_{j}Z_{j'}v\in T(j,j';v)$
			for $j,j'\in\mathbb{Z}\backslash2\mathbb{Z}$ such that $j>j'$.
		\end{prop}
		\begin{proof}
			Assigning $\alpha=\beta=\alpha_{4}$ and $k=2$ in (\ref{thrm}), 
			and comparing the coefficients of
			$\zeta_{1}^{j}\zeta_{2}^{j'}$ and $\zeta_{1}^{j+1}\zeta_{2}^{j'-1}$ respectively, 
			we get
			\begin{align}
				&(c(\alpha_{4},\alpha_{4},0)Z_{j}Z_{j'}
				-h_{1}(3,j-2j',-2j+j')
				Z_{j+j'}(\beta_{2})\notag\\*
				&
				-h_{1}(1,j,j')
				Z_{j+j'}(\beta_{3}))v
				\in{T(j,j';v)},\label{7_1}\\
				&(c(\alpha_{4},\alpha_{4},1)Z_{j}Z_{j'}
				-h_{1}(3,j-2j'+3,-2j+j'-3)
				Z_{j+j'}(\beta_{2})
				\notag\\*
				&
				-h_{1}(1,j+1,j'-1)
				Z_{j+j'}(\beta_{3}))v
				\in{T(j,j';v)}.\label{7_2}
			\end{align}
			From Proposition \ref{pree} (c), we have 
			\begin{align*}
				Z_{j+j'}(\beta_{2})v=Z_{(j+1)+(j'-1)}(\beta_{2})v\in{T((j+1)-1,(j'-1)+1;v)}=T(j,j';v).
			\end{align*}
			Thus, 
			it is sufficient
			to show
			the following determinant of the matrix consisting of the coefficients of 
			$Z_{j}(\alpha_{4})Z_{j'}(\alpha_{4})$ and $Z_{j+j'}(\beta_{3})$
			in (\ref{7_1}) and (\ref{7_2}) is not 0.
			\begin{align*}
				\begin{vmatrix}
					c(\alpha_{4},\alpha_{4},0) & c(\alpha_{4},\alpha_{4},1)\\
					h_{1}(1,j,j')& h_{1}(1,j+1,j'-1)
				\end{vmatrix}
				&=\frac{1}{14}\varepsilon(\nu\alpha_{1},\alpha_{1})
				\omega^{-j-1}(1-\omega^{j-j'+1})(1+\omega).
			\end{align*}
			As $j-j'\in2\mathbb{Z}$, 
			we have ${j-j'}\not\equiv{-1}\textrm{ (mod } 14)$ and $1-\omega^{j-j'+1}\neq0$. 
			Therefore, we obtain $Z_{j}Z_{j'}v\in{T(j,j';v)}$
			for $j,j'\in\mathbb{Z}\backslash2\mathbb{Z}$ such that $j>j'$.
		\end{proof}

		\begin{prop}
			We have $Z_{j}Z_{j}v\in T(j,j;v)$
			for $j\in\mathbb{Z}\backslash2\mathbb{Z}$. 
		\end{prop}
		\begin{proof}
			By assigning $\alpha=\beta=\alpha_{4}$ and $k=2$ in (\ref{thrm}), 
			and comparing the coefficients of
			$\zeta_{1}^{j+1}\zeta_{2}^{j-1}$, 
			$\zeta_{1}^{j+2}\zeta_{2}^{j-2}$ and $\zeta_{1}^{j+3}\zeta_{2}^{j-3}$ respectively, 
			we get
			\begin{align}
				&(c(\alpha_{4},\alpha_{4},1)Z_{j}Z_{j}
				-
				h_{1}(3,-j+3,-j-3)
				Z_{2j}(\beta_{2})\notag\\*
				&
				-
				h_{1}(1,j+1,j-1)
				Z_{2j}(\beta_{3}))v\in{T(j,j;v)},\label{7_3}\\
				&(c(\alpha_{4},\alpha_{4},0)Z_{j+2}Z_{j-2}
				+c(\alpha_{4},\alpha_{4},2)Z_{j}Z_{j}\notag\\*
				&
				-
				h_{1}(3,-j+6,-j-6)
				Z_{2j}(\beta_{2})
				-
				h_{1}(1,j+2,j-2)
				Z_{2j}(\beta_{3}))v\in{T(j,j;v)},\label{7_4}\\
				&(c(\alpha_{4},\alpha_{4},1)Z_{j+2}Z_{j-2}
				+c(\alpha_{4},\alpha_{4},3)Z_{j}Z_{j}\notag\\*
				&
				-
				h_{1}(3,-j+9,-j-9)
				Z_{2j}(\beta_{2})
				-
				h_{1}(1,j+3,j-3)
				Z_{2j}(\beta_{3}))v\in{T(j,j;v)}\label{7_5}.
			\end{align}
			We also have 
			\begin{align*}
				Z_{2j}(\beta_{2})v=Z_{(j+1)+(j-1)}(\beta_{2})v\in{T((j+1)-1,(j-1)+1;v)}=T(j,j;v)
			\end{align*}
			from Proposition \ref{pree}, 
			thus we have to calculate the determinant of the matrix
			consisting of the coefficients of
			$Z_{j+2}Z_{j-2}$, 
			$Z_{j}Z_{j}$ and $Z_{2j}(\beta_{3})$
			in (\ref{7_3}), (\ref{7_4}), and (\ref{7_5}).
			\begin{align*}
				&
				\begin{vmatrix}
					0 & c(\alpha_{4},\alpha_{4},0) & c(\alpha_{4},\alpha_{4},1)\\
					c(\alpha_{4},\alpha_{4},1) & c(\alpha_{4},\alpha_{4},2) & c(\alpha_{4},\alpha_{4},3)\\
					h_{1}(1,j+1,j-1) & h_{1}(1,j+2,j-2) & h_{1}(1,j+3,j-3)
				\end{vmatrix}\\
				&
				=h_{1}(1,j+1,j-1)
				(-2\omega^{5}+\omega^{4}-\omega^{3}+2\omega^{2}+2)\neq0.
			\end{align*}
			Thus, 
			we obtain ${Z_{j}Z_{j}}v\in{T(j,j;v)}$ for $j\in\mathbb{Z}\backslash2\mathbb{Z}$.
		\end{proof}
		
		\begin{prop}
			We have $Z_{i}Z_{i}v\in{{T}(i,i;v)}$ for $i\in2\mathbb{Z}$. 
		\end{prop}
		
		\begin{proof}
			First,
			we show $Z_{2j+2}(\beta_{2})v\in{T(j+1,j+1;v)}$
			for $j\in\mathbb{Z}\backslash2\mathbb{Z}$.
			By assigning $\alpha=\beta=\alpha_{4}$ and $k=2$ in (\ref{thrm}), 
			and comparing the coefficients of $\zeta_{1}^{j+2}\zeta_{2}^{j}$, 
			$\zeta_{1}^{j+3}\zeta_{2}^{j-1}$, 
			$\zeta_{1}^{j+4}\zeta_{2}^{j-2}$ and $\zeta_{1}^{j+5}\zeta_{2}^{j-3}$ respectively, 
			we get
			\begin{align}
				&(c(\alpha_{4},\alpha_{4},0)Z_{j+2}Z_{j}
				-
				h_{1}(3,-j+2,-j-4)
				Z_{2j+2}(\beta_{2})\notag\\*
				&-
				h_{1}(1,j+2,j)
				Z_{2j+2}(\beta_{3}))v
				\in{T(j+1,j+1;v)},\label{7_6}\\
				&(c(\alpha_{4},\alpha_{4},1)Z_{j+2}Z_{j}
				-
				h_{1}(3,-j+5,-j-7)
				Z_{2j+2}(\beta_{2})\notag\\*
				&-
				h_{1}(1,j+3,j-1)
				Z_{2j+2}(\beta_{3}))v
				\in{T(j+1,j+1;v)},\label{7_7}\\
				&(c(\alpha_{4},\alpha_{4},0)Z_{j+4}Z_{j-2}
				+c(\alpha_{4},\alpha_{4},2)Z_{j+2}Z_{j}
				-
				h_{1}(3,-j+8,-j-10)
				Z_{2j+2}(\beta_{2})\notag\\*
				&-
				h_{1}(1,j+4,j-2)
				Z_{2j+2}(\beta_{3}))v
				\in{T(j+1,j+1;v)},\label{7_8}\\
				&(c(\alpha_{4},\alpha_{4},1)Z_{j+4}Z_{j-2}
				+c(\alpha_{4},\alpha_{4},3)Z_{j+2}Z_{j}
				-
				h_{1}(3,-j+11,-j-13)
				Z_{2j+2}(\beta_{2})\notag\\*
				&
				-
				h_{1}(1,j+5,j-3)
				Z_{2j+2}(\beta_{3}))v
				\in{T(j+1,j+1;v)}.\label{7_9}
			\end{align}
			The determinant consisting of the coefficients of
			$Z_{j+4}Z_{j-2}$, 
			$Z_{j+2}Z_{j}$, 
			$Z_{2j+2}(\beta_{2})$, $Z_{2j+2}(\beta_{3})$
			in (\ref{7_6}),  (\ref{7_7}),  (\ref{7_8}), and  (\ref{7_9}) is,
			\begin{align*}
				&
				\begin{vmatrix}
					0 & 0 & c(\alpha_{4},\alpha_{4},0) & c(\alpha_{4},\alpha_{4},1)\\
					c(\alpha_{4},\alpha_{4},0) & c(\alpha_{4},\alpha_{4},1)
					& c(\alpha_{4},\alpha_{4},2) & c(\alpha_{4},\alpha_{4},3)\\
					h_{1}(3,-j+2,-j-4)
					& h_{1}(3,-j+5,-j-7)
					& h_{1}(3,-j+8,-j-10)
					& h_{1}(3,-j+11,-j-13)
					\\
					h_{1}(1,j+2,j)
					& h_{1}(1,j+3,j-1)
					& h_{1}(1,j+4,j-2)
					& h_{1}(1,j+5,j-3)
				\end{vmatrix}\\
				&=h_{1}(1,j+2,j)h_{1}(3,-j+2,-j-4)(-4\omega^5 + 4\omega^2 + 6)\neq0.
			\end{align*}
			Thus, for $j\in\mathbb{Z}\backslash2\mathbb{Z}$, we obtain
			\begin{align}\label{rabbit}
				Z_{2j+2}(\beta_{2})v\in{T(j+1,j+1;v)},
			\end{align}
			and next, 
			we show $Z_{i}Z_{i}v\in{{T}(i,i;v)}$ for $i\in2\mathbb{Z}$.
			By assigning $\alpha=\beta=\alpha_{1}$ and $k=2$ in (\ref{thrm}), 
			and comparing the coefficients of $\zeta_{1}^{i+1}\zeta_{2}^{i-1}$, 
			we get
			\begin{align*}
				(c(\alpha_{4},\alpha_{4},1)Z_{i}Z_{i}
				-
				h_{2}(2,2i+2,2i-2)
				Z_{2i}(\beta_{2}))v\in{T(i,i;v)}.
			\end{align*}
			We have $c(\alpha_{4},\alpha_{4},1)\neq0$, and from (\ref{rabbit}), we get
			\begin{align*}
				Z_{2i}(\beta_{2})v=Z_{2(i-1+1)}(\beta_{2})v\in{T((i-1)+1,(i-1)+1;v)}=T(i,i;v). 
			\end{align*}
			Thus, we obtain $Z_{i}Z_{i}v\in{T(i,i;v)}$ for $i\in2\mathbb{Z}$.
		\end{proof}

			The condition ${i_{l}}\le{-2}$ in the case of $L(\Lambda_{3})$
			follows from the principal character $\chi(L(\Lambda_{3}))$.
			By the RR identities,  we can see that the given sets are bases.

	\subsection{The case of type $A^{(2)}_{9}$}\label{9th}
	\noindent
		In this section, we prove Theorem \ref{A29thm}.
		The remaining conditions are as follows.
	\begin{enumerate}
		\item
			$Z_{j}Z_{j'}v\in{T(j,j';v)}$ for $j,j'\in{\mathbb{Z}\backslash2\mathbb{Z}}$
			such that $j\ge{j'}$.
		\item
			$Z_{i}Z_{i}Z_{i}v\in{T(i,i,i;v)}$ for $i\in2\mathbb{Z}$.
		\item
			$Z_{i}Z_{i}Z_{i+2}v\in{T(i,i,i+2;v)}$ and
			$Z_{i-2}Z_{i}Z_{i}v\in{T(i-2,i,i;v)}$
			for $i\in2\mathbb{Z}$.
		\item
			$Z_{i}Z_{i}Z_{i+3}v\in{T(i,i,i+3;v)}$ and
			$Z_{i-3}Z_{i}Z_{i}v\in{T(i-3,i,i;v)}$
			for $i\in2\mathbb{Z}$.
		\item
			$(i_{l-1},i_{l})\neq(-2,-2)$ in the case of $L(\Lambda_{0}+\Lambda_{1})$.
		\item
			$i_{l}\le-2$ (resp. $i_{l}\le-4$)
			in the case of $L(\Lambda_{3})$ (resp. $L(\Lambda_{5})$).
	\end{enumerate}

	\subsubsection{Adjacent $Z$-operators with odd indices}\label{conti_even}
		In this section, we verify the condition of adjacent odd parts.
		We can verify this condition as in the cases of $A^{(2)}_{5}$ and $A^{(2)}_{7}$ types.
		\begin{prop}\label{9odd}
			For $j,j'\in{\mathbb{Z}\backslash2\mathbb{Z}}$ such that $j\ge{j'}$,
			we have
			\begin{enumerate}
				\item$Z_{j}Z_{j'}v\in{T(j,j';v)}$,
				\item$Z_{j+j'}(\beta_{n})v\in{T(j,j';v)}$.
			\end{enumerate}
		\end{prop}

		\begin{proof}
			By assigning $\alpha=\beta=\alpha_{5}$ and $k=2$ in (\ref{thrm}),
			and comparing the coefficients of $\zeta_{1}^{j}\zeta_{2}^{j'}$, 
			$\zeta_{1}^{j+1}\zeta_{2}^{j'-1}$, 
			$\zeta_{1}^{j+2}\zeta_{2}^{j'-2}$
			and
			$\zeta_{1}^{j+3}\zeta_{2}^{j'-3}$ respectively,
			we get
			\begin{align}
				&(c(\alpha_{5},\alpha_{5},0)Z_{j}Z_{j'}
				-\frac{1}{18}\varepsilon(\nu\alpha_{5},\alpha_{5})
				(\omega^{-j}
				-\omega^{-j'})
				Z_{j+j'}(\beta_{4})\notag\\*
				&-\frac{1}{18}\varepsilon(\nu^{3}\alpha_{5},\alpha_{5})
				(\omega^{-j+2j'}
				-\omega^{2j-j'})Z_{j+j'}(\beta_{3})\notag\\*
				&-\frac{1}{18}\varepsilon(\nu^{5}\alpha_{5},\alpha_{5})
				(\omega^{-j+4j'}
				-\omega^{4j-j'})Z_{j+j'}(\beta_{2}))v
				\in{T(j,j';v)},\label{7_10}\\
				&(c(\alpha_{5},\alpha_{5},1)Z_{j}Z_{j'}
				-\frac{1}{18}\varepsilon(\nu\alpha_{5},\alpha_{5})
				(\omega^{-j-1}-\omega^{-j'+1})
				Z_{j+j'}(\beta_{4})\notag\\*
				&-\frac{1}{18}\varepsilon(\nu^{3}\alpha_{5},\alpha_{5})
				(\omega^{-j+2j'-3}
				-\omega^{2j-j'+3})Z_{j+j'}(\beta_{3})\notag\\*
				&-\frac{1}{18}\varepsilon(\nu^{5}\alpha_{5},\alpha_{5})
				(\omega^{-j+4j'-5}-\omega^{4j-j'+5})Z_{j+j'}(\beta_{2}))v
				\in{T(j,j';v)},\label{7_11}\\
				&(c(\alpha_{5},\alpha_{5},2)Z_{j+2}Z_{j'-2}
				+c(\alpha_{5},\alpha_{5},0)Z_{j}Z_{j'}\notag\\*
				&-\frac{1}{18}\varepsilon(\nu\alpha_{5},\alpha_{5})
				(\omega^{-j-2}
				-\omega^{-j'+2})
				Z_{j+j'}(\beta_{4})\notag\\*
				&-\frac{1}{18}\varepsilon(\nu^{3}\alpha_{5},\alpha_{5})
				(\omega^{-j+2j'-6}
				-\omega^{2j-j'+6})Z_{j+j'}(\beta_{3})\notag\\*
				&-\frac{1}{18}\varepsilon(\nu^{5}\alpha_{5},\alpha_{5})
				(\omega^{-j+4j'-10}
				-\omega^{4j-j'+10})Z_{j+j'}(\beta_{2}))v
				\in{T(j,j';v)},\label{7_12}\\
				&(c(\alpha_{5},\alpha_{5},3)Z_{j+2}Z_{j'-2}
				+c(\alpha_{5},\alpha_{5},1)Z_{j}Z_{j'}\notag\\*
				&-\frac{1}{18}\varepsilon(\nu\alpha_{5},\alpha_{5})
				(\omega^{-j-3}
				-\omega^{-j'+3})
				Z_{j+j'}(\beta_{4})\notag\\*
				&-\frac{1}{18}\varepsilon(\nu^{3}\alpha_{5},\alpha_{5})
				(\omega^{-j+2j'-9}
				-\omega^{2j-j'+9})Z_{j+j'}(\beta_{3})\notag\\*
				&-\frac{1}{18}\varepsilon(\nu^{5}\alpha_{5},\alpha_{5})
				(\omega^{-j+4j'-15}
				-\omega^{4j-j'+15})Z_{j+j'}(\beta_{2}))v
				\in{T(j,j';v)}.\label{7_13}
			\end{align}
			From Proposition \ref{pree}, we have
			\begin{align*}
				Z_{j+j'}(\beta_{2})v=Z_{(j+1)+(j'-1)}(\beta_{2})v\in{T((j+1)-1,(j'-1)+1;v)}=T(j,j';v).
			\end{align*}
			Thus, we have to show that the following determinant of the matrix, 
			consisting of the coefficients of $Z_{j+2}Z_{j'-2}$, 
			$Z_{j}Z_{j'}$, 
			$Z_{j+j'}(\beta_{3})$ and $Z_{i+i'}(\beta_{4})$
			in (\ref{7_10}), (\ref{7_11}), (\ref{7_12}), and (\ref{7_13}), is not 0.
			\begin{align*}
				&\begin{vmatrix}
					0 & 0 & c(\alpha_{5},\alpha_{5},0) & c(\alpha_{5},\alpha_{5},1)\\
					c(\alpha_{5},\alpha_{5},0) & c(\alpha_{5},\alpha_{5},1)
					& c(\alpha_{5},\alpha_{5},2) & c(\alpha_{5},\alpha_{5},3)\\
					\omega^{-j}-\omega^{-j'} & \omega^{-j-1}-\omega^{-j'+1}
					& \omega^{-j-2}-\omega^{-j'+2} & \omega^{-j-3}-\omega^{-j'+3}\\
					\omega^{-j+2j'}-\omega^{2j-j'} & \omega^{-j+2j'-3}-\omega^{2j-j'+3} 
					& \omega^{-j+2j'-6}-\omega^{2j-j'+6} & \omega^{-j+2j'-9}-\omega^{2j-j'+9}
				\end{vmatrix}
				\\*
				&=\omega^{-2j+2j'}
				\begin{vmatrix}
					0 & 0 & 1 & -1\\ 
					1 & -1 & \frac{1}{2} & -\frac{1}{2}\\ 
					1-\omega^{j-j'} & \omega^{-1}-\omega^{j-j'+1}
					& \omega^{-2}-\omega^{j-j'+2} & \omega^{-3}-\omega^{j-j'+3}\\ 
					1-\omega^{3j-3j'} & \omega^{-3}-\omega^{3j-3j'+3} 
					& \omega^{-6}-\omega^{3j-3j'+6} & \omega^{-9}-\omega^{3j-3j'+9} 
				\end{vmatrix}.
			\end{align*}
			From the same reason stated in the proof of Proposition \ref{pree}, 
			we only have to calculate in the cases of $j-j'=2,4$.
			If $j-j'=2$, this determinant equals to
			$\omega^{-4}(3\omega^5 + 6\omega^4 + 6\omega^3 + 3\omega^2 - 3)(\neq0)$,
			and if $j-j'=4$, this determinant equals to
			$\omega^{-6}(6\omega^5 + 6\omega^4 - 6\omega^2 - 6\omega - 6)(\neq0)$.
	
	We can prove $Z_{j}Z_{j}v\in T(j,j;v)$ for $j\in\mathbb{Z}\backslash2\mathbb{Z}$ similarly,
	and we omit it here. 
		\end{proof}

	\subsubsection{Three contiguous $Z$-operators with even indices}\label{3contie}
		In this section and next, 
		we verify the condition of three contiguous parts, 
		and in this section,  we focus on the cases of even indices only.


		\begin{dfn}
			For $i,j\in\mathbb{Z}$, define $T''(i,j;v)\subseteq{T(i,j;v)}$ by
			\begin{align*}
				T''(i,j;v)
				=\SPAN\{Z_{i-2k}Z_{j+2k}v,\ Z_{l}v,\ v\mid k,l\in\mathbb{Z},\ k\ge1\}.
			\end{align*}
		\end{dfn}


		\begin{lem} \label{lem1}
			For $i\in2\mathbb{Z}$, we have
			\begin{align}
				&(d(i)Z_{i}Z_{i}
				-
				Z_{2i}(\beta_{2}))v
				\in{T(i,i;v)},\label{lem1T}\\
				&(d(i)Z_{i}Z_{i}
				-
				Z_{2i}(\beta_{2}))v
				\in{T''(i,i;v)},\label{lem1T''}
			\end{align}
			where $d(i)$ is defined by
			\begin{align*}
				d(i)
				=(2\omega^5 - 2\omega^3 - 2\omega^2 + 2\omega + 1)
				(18\omega^{2i}/\varepsilon(\nu^{2}\alpha_{1},\alpha_{1})).
			\end{align*}
		\end{lem}

		\begin{proof}
			It is enough to show (\ref{lem1T''}) as $T''(i,i;v)\subseteq{T(i,i;v)}$.
			Take $i\in2\mathbb{Z}$. 
			By assigning $\alpha=\beta=\alpha_{1}$ and $k=2$ in (\ref{thrm}), 
			and comparing the coefficients of $\zeta_{1}^{i+1}\zeta_{2}^{i-1}$, 
			we have
			\begin{align*}
				&(c(\alpha_{1},\alpha_{1},1)Z_{i}Z_{i}
				-\frac{1}{18}\varepsilon(\nu^{2}\alpha_{1},\alpha_{1})
				(\omega^{-2i-2}-\omega^{-2i+2})Z_{2i}(\beta_{2}))v
				\in T''(i,i;v).
			\end{align*}
			We have the consequence by dividing this formula
			by the coefficient of $Z_{2i}(\beta_{2})$.
		\end{proof}

		\begin{lem} \label{lem2}
			For $i\in2\mathbb{Z}$, we have
			\begin{align}
				&(e(i)Z_{i}Z_{i+2}-Z_{2i+2}(\beta_{2}))v
				\in{T(i,i+2;v)},\label{lem2T}\\
				&(e(i)Z_{i}Z_{i+2}-Z_{2i+2}(\beta_{2}))v
				\in{T''(i,i+2;v)},\label{lem2T''}
			\end{align}
			where $e(i)$ is defined by
			\begin{align*}
				e(i)
				=(4\omega^5 - \omega^4 - 4\omega^3 - 2\omega^2 + 4\omega + 1)
				(18\omega^{2i}/\varepsilon(\nu^{2}\alpha_{1},\alpha_{1})).
			\end{align*}
		\end{lem}

		\begin{proof}
			It is enough to show (\ref{lem2T''}).
			Take $i\in2\mathbb{Z}$. 
			By assigning $\alpha=\beta=\alpha_{1}$ and $k=2$ in (\ref{thrm}), 
			and comparing the coefficients of $\zeta_{1}^{i+2}\zeta_{2}^{i}$ and
			$\zeta_{1}^{i+3}\zeta_{2}^{i-1}$ respectively, 
			we get
			\begin{align*}
				&(c(\alpha_{1},\alpha_{1},0)Z_{i+2}Z_{i}
				+(c(\alpha_{1},\alpha_{1},2)-c(\alpha_{1},\alpha_{1},0))
				Z_{i}Z_{i+2}\\*
				&-\frac{1}{18}\varepsilon(\nu^{2}\alpha_{1},\alpha_{1})
				(\omega^{-2i-4}-\omega^{-2i})Z_{2i+2}(\beta_{2}))v
				\in{T''(i,i+2;v)},\\
				&(c(\alpha_{1},\alpha_{1},1)Z_{i+2}Z_{i}
				+c(\alpha_{1},\alpha_{1},3)
				Z_{i}Z_{i+2}\\*
				&-\frac{1}{18}\varepsilon(\nu^{2}\alpha_{1},\alpha_{1})
				(\omega^{-2i-6}-\omega^{-2i+2})Z_{2i+2}(\beta_{2}))v
				\in{T''(i,i+2;v)}.
			\end{align*}
			We obtain the consequence by vanishing $Z_{i+2}Z_{i}v$
			from these formulae.
		\end{proof}
		
		\begin{rmk}
			The formulae (\ref{lem1T''}) and (\ref{lem2T''})
			are utilized only in the proof of Proposition \ref{Butler}.
		\end{rmk}

		\begin{prop}\label{Austin}
			For $i\in2\mathbb{Z}$, we have
				$Z_{i}Z_{i}Z_{i}v
				\in{T(i,i,i;v)}$.
		\end{prop}

		\begin{proof}
%
%
			We consider the commutation relation
			between $Z_{i}(\alpha_{1})$ and $Z_{2i}(\beta_{2})$
			for showing $Z_{i}(\alpha_{1})Z_{2i}(\beta_{2})v\in{T(i,i,i;v)}$.
			By assigning $\alpha=\alpha_{1}$, $\beta=\beta_{2}$ and $k=2$ in (\ref{thrm}), 	
			comparing the coefficients of $\zeta_{1}^{i+1}\zeta_{2}^{2i-1}$, we get
		        \begin{align}
				&\sum_{p\ge1,\ p\in\mathbb{Z}\backslash2\mathbb{Z}}
				(c(\alpha_{1},\beta_{2},p)
				Z_{i+1-p}(\alpha_{1})
				Z_{2i-1+p}(\beta_{2})\notag\\*
				&-
				c(\beta_{2},\alpha_{1},p)
				Z_{2i-1-p}(\beta_{2})
				Z_{i+1+p}(\alpha_{1})
				)\notag\\*
				&=d_{1}(i+1,2i-1)Z_{3i}(\alpha_{1})
				+d_{2}(i+1,2i-1)Z_{3i}(\beta_{3}),\label{coef1'}
			\end{align}
			where $d_{1}(a,b)$ and $d_{2}(a,b)$ are the constants
			determined by comparing the coefficients of $\zeta_{1}^{a}\zeta_{2}^{b}$.
			From proposition \ref{9odd}, 
			$Z_{3i}(\beta_{3})v$ is spanned by the $Z$-monomials
			with length shorter than 3,  thus we have
			\begin{align}\label{3i}
				Z_{3i}(\beta_{3})v\in{T(i,i,i;v)}.
			\end{align}

			From (\ref{coef1'}) and (\ref{3i}), we have
			\begin{align}\label{i,i,i,1}
				&\sum_{p\ge1,\ p\in\mathbb{Z}\backslash2\mathbb{Z}}
				(c(\alpha_{1},\beta_{2},p)
				Z_{i+1-p}(\alpha_{1})Z_{2i-1+p}(\beta_{2})\notag\\*
				&-
				c(\beta_{2},\alpha_{1},p)
				Z_{2i-1-p}(\beta_{2})Z_{i+1+p}(\alpha_{1})
				)v
				\in{T(i,i,i;v)}.
			\end{align}

			Now,  let $p\in\mathbb{Z}\backslash2\mathbb{Z}$,  and
			we prove that $Z_{i+1-p}(\alpha_{1})Z_{2i-1+p}(\beta_{2})v$
			for $p\ge3$
			and
			$Z_{2i-1-p}(\beta_{2})Z_{i+1+p}(\alpha_{1})v$ for $p\ge1$
			belong to $T(i,i,i;v)$,
			classifying the cases by whether $p-1$ is a multiple of 4 or not.
			Then,  from (\ref{i,i,i,1}) we obtain $Z_{i}(\alpha_{1})Z_{2i}(\beta_{2})v\in{T(i,i,i;v)}$,
			and rewrite $Z_{i}(\alpha_{1})Z_{2i}(\beta_{2})v$
			via the element
			$Z_{i}(\beta_{1})Z_{i}(\alpha_{1})Z_{i}(\alpha_{1})v$.


			Set $p\ge1,\ p\in\mathbb{Z}\backslash2\mathbb{Z}$ and $p-1\in4\mathbb{Z}$. 
			Then $p-1$ can be denoted by $4n\ (n\ge0)$,
			and we have
			$
				Z_{i+1-p}(\alpha_{1})
				Z_{2i-1+p}(\beta_{2})
				=
				Z_{i-4n}(\alpha_{1})
				Z_{2i+4n}(\beta_{2}).
			$
			From (\ref{lem1T}), we obtain
			\begin{align*}
				&(Z_{i-4n}(\alpha_{1})
				Z_{2i+4n}(\beta_{2})\\*
				&-
				d(i+2n)
				Z_{i-4n}
				Z_{i+2n}
				Z_{i+2n})v
				\in{T(i-4n,i+2n,i+2n;v)}.
			\end{align*}
			Thus, if $n=0$,  we have
			\begin{align} \label{i,2i}
				(Z_{i}(\alpha_{1})
				Z_{2i}(\beta_{2})
				-
				d(i)Z_{i}
				Z_{i}Z_{i})v
				\in{T(i,i,i;v)},
			\end{align}
			and if $n\ge1$, 
			we have
			\begin{align} \label{i-4n,2i+4n}
				Z_{i-4n}(\alpha_{1})
				Z_{2i+4n}(\beta_{2})v
				\in{T(i,i,i;v)}.
			\end{align}
			We also have $p+1=4n+2$,  and 
			$
				Z_{2i-1-p}(\beta_{2})
				Z_{i+1+p}(\alpha_{1})
				=
				Z_{2i-4n-2}(\beta_{2})
				Z_{i+4n+2}(\alpha_{1}).
			$
			From Proposition \ref{pree}, we have
			\begin{align} \label{2i-4n-2,i+4n+2}
				Z_{2i-4n-2}(\beta_{2})
				Z_{i+4n+2}(\alpha_{1})v
				\in{T(i-2n-1,i-2n-1,i+4n+2;v)}\subseteq{T(i,i,i;v)}.
			\end{align}

    
			Likewise, let $p\in\mathbb{Z}\backslash2\mathbb{Z}$ such that $p\ge1$
			and $p-1\notin4\mathbb{Z}$. 
			Then $p-1$ can be denoted by $4n+2\ (n\ge0)$, and we have
			$
				Z_{i+1-p}(\alpha_{1})
				Z_{2i-1+p}(\beta_{2})
				=
				Z_{i-4n-2}(\alpha_{1})
				Z_{2i+4n+2}(\beta_{2}).
			$ 
			From Proposition \ref{pree},  we have
			\begin{align} \label{i-4n-2,2i+4n+2}
					Z_{i-4n-2}(\alpha_{1})
					Z_{2i+4n+2}(\beta_{2})v
					\in{T(i-4n-2,i+2n+1,i+2n+1;v)}
					\subseteq{T(i,i,i;v)}.
			\end{align}
			We also have $p+1=4n+4$ and
			$
				Z_{2i-1-p}(\beta_{2})
				Z_{i+1+p}(\alpha_{1})
				=
				Z_{2i-4n-4}(\beta_{2})
				Z_{i+4n+4}(\alpha_{1}).
			$ 
			From
			(\ref{lem1T}), we have
			\begin{align}
				&(Z_{2i-4n-4}(\beta_{2})
				Z_{i+4n+4}(\alpha_{1})\notag\\*
				&-
				d(i-2n-2)
				Z_{i-2n-2}
				Z_{i-2n-2}
				Z_{i+4n+4})v
				\in{T(i-2n-2,i-2n-2,i+4n+4;v)}
				\subseteq{T(i,i,i;v)}.\label{star}
			\end{align}
			From (\ref{star}) and the fact that
			$Z_{i-2n-2}
			Z_{i-2n-2}
			Z_{i+4n+4}v\in{T(i,i,i;v)}$,  we have
			\begin{align} \label{2i-4n-4,i+4n+4}
				Z_{2i-4n-4}(\beta_{2})
				Z_{i+4n+4}(\alpha_{1})v
				\in{T(i,i,i;v)}.
			\end{align}
    
			Thus, from (\ref{i-4n,2i+4n}) and (\ref{i-4n-2,2i+4n+2}), 
			for $p\in\mathbb{Z}\backslash2\mathbb{Z}$ such that $p\ge3$, we get
			\begin{align}\label{garden}
				Z_{i+1-p}(\alpha_{1})Z_{2i-1+p}(\beta_{2})v\in{T(i,i,i;v)},
			\end{align}
			and from (\ref{2i-4n-2,i+4n+2}) and (\ref{2i-4n-4,i+4n+4}), 
			for $p\in\mathbb{Z}\backslash2\mathbb{Z}$ such that $p\ge1$, we get 
			\begin{align}\label{garden'}
				Z_{2i-1-p}(\beta_{2})Z_{i+1+p}(\alpha_{1})v\in{T(i,i,i;v)}.
			\end{align}
			Therefore, from (\ref{i,i,i,1}), (\ref{garden}), and (\ref{garden'}),  we have
			\begin{align}\label{garden''}
				c(\alpha_{1},\beta_{2},1)Z_{i}(\alpha_{1})Z_{2i}(\beta_{2})v\in{T(i,i,i;v)}. 
			\end{align}
			From (\ref{garden''}) and (\ref{i,2i}),
			we get
			\begin{align*}
				c(\alpha_{1},\beta_{2},1)d(i)
				Z_{i}Z_{i}Z_{i}v\in{T(i,i,i;v)}.
			\end{align*}
			As $c(\alpha_{1},\beta_{2},1)d(i)
			\neq0$, 
			we obtain $Z_{i}Z_{i}Z_{i}v\in{T(i,i,i;v)}$.
		\end{proof}

		\begin{prop}
			For $i\in2\mathbb{Z}$, we have
			\begin{align*}
				Z_{i}Z_{i}Z_{i+2}v
				\in{T(i,i,i+2;v)},\qquad
				Z_{i-2}Z_{i}Z_{i}v
				\in{T(i-2,i,i;v)}.
			\end{align*}
		\end{prop}

		\begin{proof}
			We prove only $Z_{i}Z_{i}Z_{i+2}v\in{T(i,i,i+2;v)}$
			for $i\in2\mathbb{Z}$.
			By assigning $\alpha=\alpha_{1}$, $\beta=\beta_{2}$ and $k=2$ in (\ref{thrm}), 
			comparing the coefficients of $\zeta_{1}^{i+n+2}\zeta_{2}^{2i-n}$ (for $2\le{n}\le5$), 
			we get
			\begin{align}\label{9_1}
				&\sum_{p\in\mathbb{Z}}
				(c(\alpha_{1},\beta_{2},p)
				Z_{i+n+2-p}(\alpha_{1})
				Z_{2i-n+p}(\beta_{2})
				-
				c(\beta_{2},\alpha_{1},p)
				Z_{2i-n-p}(\beta_{2})
				Z_{i+n+2+p}(\alpha_{1})
				)\notag\\*
				&=d_{1}(i+n+2,2i-n)Z_{3i+2}(\alpha_{1})
				+d_{2}(i+n+2,2i-n)Z_{3i+2}(\beta_{3}).
			\end{align}
			From proposition \ref{9odd} (2), 
			we have $Z_{3i+2}(\beta_{3})v\in{T(i,i,i+2;v)}$. 
			In addition,
			for $p\in2\mathbb{Z}$ such that $p\ge4$, 
			we have $Z_{i+2-p}(\alpha_{1})Z_{2i+p}(\beta_{2})v\in{T(i,i,i+2;v)}$, 
			and for $p\in2\mathbb{Z}$ such that 
			$p\ge2$, we have $Z_{2i-p}(\beta_{2})Z_{i+2+p}(\alpha_{1})v\in{T(i,i,i+2;v)}$
			from the same reason explained in the proof of Proposition \ref{Austin}.
			Thus, from (\ref{9_1})
			we get
			\begin{align}
				&(c(\alpha_{1},\beta_{2},0)
				Z_{i+4}(\alpha_{1})
				Z_{2i-2}(\beta_{2})
				+
				c(\alpha_{1},\beta_{2},2)
				Z_{i+2}(\alpha_{1})
				Z_{2i}(\beta_{2})\notag\\*
				&+
				c(\alpha_{1},\beta_{2},4)
				Z_{i}(\alpha_{1})
				Z_{2i+2}(\beta_{2}))v
				\in{T(i,i,i+2;v)},\label{9_9}\\
				&(c(\alpha_{1},\beta_{2},1)
				Z_{i+4}(\alpha_{1})
				Z_{2i-2}(\beta_{2})
				+
				c(\alpha_{1},\beta_{2},3)
				Z_{i+2}(\alpha_{1})
				Z_{2i}(\beta_{2})\notag\\*
				&
				+c(\alpha_{1},\beta_{2},5)
				Z_{i}(\alpha_{1})
				Z_{2i+2}(\beta_{2}))v
				\in{T(i,i,i+2;v)},\label{9_10}\\
				&(c(\alpha_{1},\beta_{2},0)
				Z_{i+6}(\alpha_{1})
				Z_{2i-4}(\beta_{2})
				+
				c(\alpha_{1},\beta_{2},2)
				Z_{i+4}(\alpha_{1})
				Z_{2i-2}(\beta_{2})\notag\\*
				&
				+
				c(\alpha_{1},\beta_{2},4)
				Z_{i+2}(\alpha_{1})
				Z_{2i}(\beta_{2})\notag\\*
				&
				+c(\alpha_{1},\beta_{2},6)
				Z_{i}(\alpha_{1})
				Z_{2i+2}(\beta_{2}))v
				\in{T(i,i,i+2;v)},\label{9_11}\\
				&(c(\alpha_{1},\beta_{2},1)
				Z_{i+6}(\alpha_{1})
				Z_{2i-4}(\beta_{2})
				+
				c(\alpha_{1},\beta_{2},3)
				Z_{i+4}(\alpha_{1})
				Z_{2i-2}(\beta_{2})\notag\\*
				&
				+
				c(\alpha_{1},\beta_{2},5)
				Z_{i+2}(\alpha_{1})
				Z_{2i}(\beta_{2})\notag\\*
				&
				+c(\alpha_{1},\beta_{2},7)
				Z_{i}(\alpha_{1})
				Z_{2i+2}(\beta_{2}))v
				\in{T(i,i,i+2;v)},\label{9_12}
			\end{align}
			The determinant of the matrix
			consisting of the coefficients of $Z_{i+6}(\alpha_{1})Z_{2i-4}(\beta_{2})$, 
			$Z_{i+4}(\alpha_{1})Z_{2i-2}(\beta_{2})$, 
			$Z_{i+2}(\alpha_{1})Z_{2i}(\beta_{2})$ and
			$Z_{i}(\alpha_{1})Z_{2i+2}(\beta_{2})$
			in (\ref{9_9}), (\ref{9_10}), (\ref{9_11}), and (\ref{9_12}) is,
			\begin{align*}
				&\begin{vmatrix}
					0 & 0
					& c(\alpha_{1},\beta_{2},0) & c(\alpha_{1},\beta_{2},1)\\
					c(\alpha_{1},\beta_{2},0) & c(\alpha_{1},\beta_{2},1)
					& c(\alpha_{1},\beta_{2},2) & c(\alpha_{1},\beta_{2},3)\\
					c(\alpha_{1},\beta_{2},2) & c(\alpha_{1},\beta_{2},3)
					& c(\alpha_{1},\beta_{2},4) & c(\alpha_{1},\beta_{2},5)\\
					c(\alpha_{1},\beta_{2},4) & c(\alpha_{1},\beta_{2},5)
					& c(\alpha_{1},\beta_{2},6) & c(\alpha_{1},\beta_{2},7)
				\end{vmatrix}
				\notag\\*
				&=3\omega^5 - 4\omega^4 + 2\omega^3 - 3\omega^2 + 2\omega + 2
				\neq0.
			\end{align*}
			Thus, we have $Z_{i}(\alpha_{1})Z_{2i+2}(\beta_{2})v\in{T(i,i,i+2;v)}$. 
			From this and (\ref{lem2T}), 
			we get $e(i)Z_{i}Z_{i}Z_{i+2}v\in{T(i,i,i+2;v)}$,
			and obtain $Z_{i}Z_{i}Z_{i+2}v\in{T(i,i,i+2;v)}$.
		\end{proof}


	\subsubsection{Three contiguous $Z$-operators including an odd part}\label{conti_odd}
		In this section, we prove $Z_{i}Z_{i}Z_{i+3}v\in{T(i,i,i+3;v)}$
		and $Z_{i-3}Z_{i}Z_{i}v\in{T(i-3, i, i;v)}$ for $i\in2\mathbb{Z}$.
		\begin{lem} \label{lem3}
			\begin{enumerate}
				\item
					For $(i, j)\in(2\mathbb{Z},\mathbb{Z}\backslash2\mathbb{Z})$
					or $(\mathbb{Z}\backslash2\mathbb{Z},2\mathbb{Z})$
					such that $i-j\ge-1$, we have
					\begin{align*}
						Z_{i}Z_{j}v
						\in
						T\left(\frac{i+j-1}{2},\frac{i+j+1}{2};v\right).
					\end{align*}
				\item
					For $i\in2\mathbb{Z}$, we have
            				\begin{align*}
						(Z_{i+2}Z_{i+1}
						+\left(\frac{1}{2}\omega^3 + \frac{1}{2}\omega^2 - \frac{1}{2}\right)
						Z_{i}Z_{i+3})v
						\in
						T(i,i+3;v).
					\end{align*}
			\end{enumerate}
		\end{lem}
		\begin{proof}
			(1)
			Define $T'(a,b;v)\subseteq{T(a,b;v)}$ for $a,b\in\mathbb{Z}$ by
			\begin{align*}
				T'(a,b;v)=\SPAN\{Z_{a-l}Z_{b+l}v,\ Z_{l'}v,\ v\mid l,l'\in\mathbb{Z},\ l\ge1\}.
			\end{align*}
			It is enough to show $Z_{i}Z_{j}v\in{T'(i,j;v)}$
			for given $i,j$.
%
%
%
%
			Set $i\in2\mathbb{Z}$ and $j\in\mathbb{Z}\backslash2\mathbb{Z}$ such that $i>j$.
			By assigning $\alpha=\alpha_{1}$, $\beta=\alpha_{5}$ and $k=2$ in (\ref{thrm}), 
			and comparing the coefficients of
			$\zeta_{1}^{i}\zeta_{2}^{j}$ and $\zeta_{1}^{i+1}\zeta_{2}^{i}$ respectively,  we get
			\begin{align*}
				&c(\alpha_{1},\alpha_{5},0)Z_{i}Z_{j}v
				=Z_{i}Z_{j}v\\*
				&=(c(\alpha_{5},\alpha_{1},0)Z_{j}Z_{i}\\*
				&-
				\sum_{p\ge2,\ p\in2\mathbb{Z}}
				(c(\alpha_{1},\alpha_{5},p)Z_{i-p}Z_{j+p}
				-c(\alpha_{5},\alpha_{1},p)Z_{j-p}Z_{i+p})\\*
				&+d_{3}(i,j)Z_{i+j})v
				\in
				T'(i,j;v),\\
				&c(\alpha_{1},\alpha_{5},1)Z_{i}Z_{i+1}v
				=(-\omega^{3}+\omega)Z_{i}Z_{i+1}v\\*
				&=(c(\alpha_{5},\alpha_{1},1)Z_{i-1}Z_{i+2}\\*
				&-\sum_{p\ge3,\ p\in\mathbb{Z}\backslash2\mathbb{Z}}
				(c(\alpha_{1},\alpha_{5},p)Z_{i+1-p}Z_{i+p}
				-c(\alpha_{5},\alpha_{1},p)Z_{i-p}Z_{i+1+p})\\*
				&+d_{3}(i+1,i)Z_{2i+1})v\in{T'(i,i+1;v)},
			\end{align*}
			where $d_{3}(a,b)$ is the constant
			determined by comparing the coefficients of $\zeta_{1}^{a}\zeta_{2}^{b}$.


			Set $i\in\mathbb{Z}\backslash2\mathbb{Z}$ and $j\in2\mathbb{Z}$ such that $i>j$. 
			Then we have
			\begin{align*}
				&c(\alpha_{5},\alpha_{1},0)Z_{i}Z_{j}v
				=Z_{i}Z_{j}v\\*
				&=(c(\alpha_{1},\alpha_{5},0)Z_{j}Z_{i}\\*
				&+\sum_{p\ge2,\ p\in2\mathbb{Z}}
				(c(\alpha_{1},\alpha_{5},p)Z_{j-p}Z_{i+p}
				-c(\alpha_{5},\alpha_{1},p)Z_{i-p}Z_{j+p})\\*
				&-d_{3}(j,i)Z_{i+j})v\in{T'(i,j;v)},\\
				&c(\alpha_{5},\alpha_{1},1)Z_{i}Z_{i+1}v
				=(-\omega^{-3}+\omega^{-1})Z_{i}Z_{i+1}v\\*
				&=(c(\alpha_{1},\alpha_{5},1)Z_{i-1}Z_{i+2}\\*
				&
				-\sum_{p\ge3\ p\in\mathbb{Z}\backslash2\mathbb{Z}}
				(c(\alpha_{1},\alpha_{5},p)Z_{i-p}Z_{i+1+p}
				-c(\alpha_{5},\alpha_{1},p)Z_{i+1-p}Z_{i+p})\\*
				&
				-d_{3}(i,i+1)Z_{2i+1})v\in{T'(i,i+1;v)}.
			\end{align*}
			Therefore we obtain the result.

			(2) Set $i\in2\mathbb{Z}$.
%
%
			By assigning $\alpha=\alpha_{1}$, $\beta=\alpha_{5}$ and $k=2$ in (\ref{thrm}), 
			and comparing the coefficients of $\zeta_{1}^{i+2}\zeta_{2}^{i+1}$, we have
			\begin{align}\label{2.eq1}
				&(c(\alpha_{1},\alpha_{5},0)
				Z_{i+2}Z_{i+1}
				-
				c(\alpha_{5},\alpha_{1},0)
				Z_{i+1}Z_{i+2}
				+
				c(\alpha_{1},\alpha_{5},2)
				Z_{i}Z_{i+3})v
				\notag\\*
				&=
				(Z_{i+2}Z_{i+1}
				-
				Z_{i+1}Z_{i+2}
				+
				c(\alpha_{1},\alpha_{5},2)
				Z_{i}Z_{i+3})v
				\in{T(i,i+3;v)}.
			\end{align}
%
%
			Likewise, by assigning $\alpha=\alpha_{1}$, $\beta=\alpha_{5}$ and $k=2$ in (\ref{thrm}), 
			and comparing the coefficients of $\zeta_{1}^{i+1}\zeta_{2}^{i+2}$, we have
			\begin{align}\label{2.eq2}
				&(c(\alpha_{1},\alpha_{5},1)
				Z_{i}Z_{i+3}
				-
				c(\alpha_{5},\alpha_{1},1)
				Z_{i+1}Z_{i+2})v
				\in
				T(i,i+3;v)
			\end{align}
			We get the consequence by vanishing $Z_{i+1}Z_{i+2}v$ from (\ref{2.eq1})
			using (\ref{2.eq2}).
		\end{proof}

		\begin{lem} \label{lem4}
			For $i\in2\mathbb{Z}$,  we have
			\begin{align*}
				((\omega^5 - \omega^2)
				Z_{i}
				Z_{i}
				Z_{i+3}
				-
				Z_{i-1}
				Z_{i+2}
				Z_{i+2})v
				\in{T(i,i,i+3;v)}.
			\end{align*}
		\end{lem}

		\begin{proof}
			Set $i\in2\mathbb{Z}$.
			By assigning $\alpha=\alpha_{1}$, $\beta=\alpha_{5}$ and $k=2$ in (\ref{thrm}), 
			and comparing the coefficients of $\zeta_{1}^{i+1}\zeta_{2}^{i+2}$, we get
			\begin{align*}
				(c(\alpha_{1},\alpha_{5},1)Z_{i}Z_{i+3}
				-c(\alpha_{5},\alpha_{1},1)Z_{i+1}Z_{i+2})v
				\in{T(i,i+3;v)},
			\end{align*}
			and multiplying the $Z$-polynomial by $Z_{i}$ from left side, we obtain
			\begin{align}\label{inu}
				(c(\alpha_{1},\alpha_{5},1)
				Z_{i}Z_{i}Z_{i+3}
				-c(\alpha_{5},\alpha_{1},1)
				Z_{i}Z_{i+1}Z_{i+2})v\in{T(i,i,i+3;v)}.
			\end{align}
			Likewise, 
			by assigning $\alpha=\alpha_{1}$, $\beta=\alpha_{5}$ and $k=2$ in (\ref{thrm}), 
			comparing the coefficients of $\zeta_{1}^{i+1}\zeta_{2}^{i}$, 
			and multiplying the $Z$-polynomial by $Z_{i+2}$ from right side, we have
			\begin{align}\label{Bronte}
				&\sum_{p\ge1,p\in\mathbb{Z}\backslash2\mathbb{Z}}
				(c(\alpha_{1},\alpha_{5},p)
				Z_{i-p+1}Z_{i+p}Z_{i+2}\notag\\*
				&-c(\alpha_{5},\alpha_{1},p)
				Z_{i-p}Z_{i+1+p}Z_{i+2})v
				\in{T(i,i,i+3;v)}.
			\end{align}
			If we get
			\begin{align}
				Z_{i-p+1}Z_{i+p}Z_{i+2}v
				\in{T(i,i,i+3;v)},\label{Nadja}\\*
				Z_{i-p}Z_{i+1+p}Z_{i+2}v
				\in{T(i,i,i+3;v)}\label{Apple}
			\end{align}
			for $p\in\mathbb{Z}\backslash2\mathbb{Z}$ such that $p\ge3$, 
			from (\ref{Bronte}) we obtain
			\begin{align}\label{innu}
				(c(\alpha_{1},\alpha_{5},1)
				Z_{i}Z_{i+1}Z_{i+2}
				-c(\alpha_{5},\alpha_{1},1)
				Z_{i-1}Z_{i+2}Z_{i+2})v
				\in{T(i,i,i+3;v)},
			\end{align}
			and by uniting (\ref{inu}) and (\ref{innu}), we obtain the consequence.
			
			We show (\ref{Nadja}) and (\ref{Apple}).
			From Proposition \ref{pree}(1),  we have
			\begin{align*}
				Z_{i+1+p}Z_{i+2}v\in{S(0;i+p,i+3;v)},
			\end{align*}
			for $p\in\mathbb{Z}\backslash2\mathbb{Z}$ such that $p\ge3$, and
			as ${S(0;i+p,i+3;v)}\subseteq{T(i+p,i+3;v)}$,
			we get 
			\begin{align*}
				Z_{i-p}Z_{i+1+p}Z_{i+2}v
				\in{T(i-p,i+p,i+3;v)}\subseteq{T(i,i,i+3;v)}.
			\end{align*}
			Likewise, from Proposition \ref{preo},  we have
			\begin{align*}
				Z_{i+p}Z_{i+2}v\in{S(1;i+p-1,i+3;v)},
			\end{align*}
			and as ${S(1;i+p-1,i+3;v)}\subseteq{T(i+p-1,i+3;v)}$,  we get
			\begin{align*}
				Z_{i+1-p}Z_{i+p}Z_{i+2}v
				\in{T(i-p+1,i+p-1,i+3;v)}\subseteq{T(i,i,i+3;v)}.
			\end{align*}
			Therefore, we obtain the consequence.
		\end{proof}

		\begin{prop}\label{Butler}
			For $i\in2\mathbb{Z}$, we have
			\begin{align*}
				Z_{i}Z_{i}Z_{i+3}v\in{T(i,i,i+3;v)},\qquad
				Z_{i-3}Z_{i}Z_{i}v\in{T(i-3,i,i;v)}.
			\end{align*}
		\end{prop}

		\begin{proof}
			We 
			prove only $Z_{i}Z_{i}Z_{i+3}v\in{T(i,i,i+3;v)}$.
			First,  we show 
			\begin{align*}
				&(c(\alpha_{5},\beta_{2},1)Z_{i-1}(\alpha_{5})
				Z_{2i+4}(\beta_{2})
				-
				c(\beta_{2},\alpha_{5},1)Z_{2i+2}(\beta_{2})
				Z_{i+1}(\alpha_{5})\notag\\*
				&-
				c(\beta_{2},\alpha_{5},3)Z_{2i}(\beta_{2})
				Z_{i+3}(\alpha_{5}))v
				\in{T(i,i,i+3;v)}
			\end{align*}
			for $i\in2\mathbb{Z}$. 
			By assigning $\alpha=\alpha_{5}$, $\beta=\beta_{2}$ and $k=2$ in (\ref{thrm})
			and comparing the coefficients of $\zeta_{1}^{i}\zeta_{2}^{2i+3}$, we have
			\begin{align}\label{i,2i+3}
				&\sum_{p\ge1,\ p\in\mathbb{Z}\backslash2\mathbb{Z}}
				(c(\alpha_{5},\beta_{2},p)
				Z_{i-p}(\alpha_{5})
				Z_{2i+3+p}(\beta_{2})
				-
				c(\beta_{2},\alpha_{5},p)Z_{2i+3-p}(\beta_{2})
				Z_{i+p}(\alpha_{5}))v\notag\\*
				&=d_{4}(i)Z_{i+b}(\alpha_{5})v,
			\end{align}
			where $d_{4}(i)$ is a constant depend on $i$.
			Set $n\in\mathbb{Z}$. 
			If $n\in2\mathbb{Z}$, from (\ref{lem1T}), we have
			$
			Z_{2i+2n}(\beta_{2})v\in{T(i+n+1,i+n-1;v)}.
			$
			If $n\in\mathbb{Z}\backslash2\mathbb{Z}$, from (\ref{lem2T}), we have
			$
			Z_{2i+2n}(\beta_{2})v\in{T(i+n,i+n;v)}.
			$
			Thus, for $p\in\mathbb{Z}\backslash2\mathbb{Z}$ such that $p\ge3$, 
			if $3+p\in4\mathbb{Z}+2$, we obtain
			\begin{align} \label{i-p,2i+3+p,1}
				Z_{i-p}(\alpha_{5})Z_{2i+3+p}(\beta_{2})v
				\in{T\left(i-p,i+\frac{3+p}{2},i+\frac{3+p}{2};v\right)}\subseteq{T(i,i,i+3;v)},
			\end{align}
			and if $3+p\in4\mathbb{Z}$, we have
			\begin{align} \label{i-p,2i+3+p,2}
				Z_{i-p}(\alpha_{5})Z_{2i+3+p}(\beta_{2})v
				\in{T\left(i-p,i+\frac{3+p}{2}+1,i+\frac{3+p}{2}-1;v\right)}\subseteq{T(i,i,i+3;v)}.
			\end{align}
			For $p\in\mathbb{Z}\backslash2\mathbb{Z}$ such that $p\ge5$, 
			similarly we get 
			\begin{align} \label{2i+3-p,i+p}
				Z_{2i+3-p}(\beta_{2})Z_{i+p}(\alpha_{5})v\in{T(i,i,i+3;v)}.
			\end{align}
			From (\ref{i,2i+3}),(\ref{i-p,2i+3+p,1}),(\ref{i-p,2i+3+p,2}) and (\ref{2i+3-p,i+p}), 
			we have
			\begin{align}\label{middle}
				&(c(\alpha_{5},\beta_{2},1)Z_{i-1}(\alpha_{5})
				Z_{2i+4}(\beta_{2})
				-
				c(\beta_{2},\alpha_{5},1)Z_{2i+2}(\beta_{2})
				Z_{i+1}(\alpha_{5})\notag\\*
				&-
				c(\beta_{2},\alpha_{5},3)Z_{2i}(\beta_{2})
				Z_{i+3}(\alpha_{5}))v
				\in{T(i,i,i+3;v)}.
			\end{align}

			Now we prove
			\begin{align*}
				&(Z_{i-1}(\alpha_{5})Z_{2i+4}(\beta_{2})
				-d(i+2)Z_{i-1}
				Z_{i+2}Z_{i+2})v
				\in
				T(i,i,i+3;v),\\
				&(Z_{2i+2}(\beta_{2})Z_{i+1}(\alpha_{5})
				+
				e(i)
				\left(\frac{1}{2}\omega^3 + \frac{1}{2}\omega^2 - \frac{1}{2}\right)
				Z_{i}
				Z_{i}
				Z_{i+3})v
				\in
				{T(i,i,i+3;v)},\\
				&(Z_{2i}(\beta_{2})Z_{i+3}(\alpha_{5})
				-d(i)Z_{i}Z_{i}
				Z_{i+3})v
				\in
				T(i,i,i+3;v).
			\end{align*}
%
%
			From (\ref{lem1T''}), we have
			\begin{align*}
				(Z_{2i+4}(\beta_{2})
				-d(i+2)Z_{i+2}Z_{i+2})v
				\in{T''(i+2,i+2;v)},
			\end{align*}
			thus we have
			\begin{align}\label{i-1,2i+4}
				&(Z_{i-1}(\alpha_{5})Z_{2i+4}(\beta_{2})\notag\\*
				&-d(i+2)Z_{i-1}
				Z_{i+2}Z_{i+2})v
				\in{Z_{i-1}T''(i+2,i+2;v)}
				\subseteq
				T(i,i,i+3;v).
			\end{align}


			Likewise, from (\ref{lem2T''}), we have
			\begin{align*}
				(Z_{2i+2}(\beta_{2})
				-
				e(i)Z_{i}Z_{i+2})v
				\in
				T''(i,i+2;v),
			\end{align*}
			thus we have
			\begin{align}\label{2i+2,i+1}
				&(Z_{2i+2}(\beta_{2})Z_{i+1}(\alpha_{5})\notag\\*
				&-
				e(i)Z_{i}Z_{i+2}
				Z_{i+1})v
				\in
				\SPAN\{Z_{i-2k}Z_{i+2+2k}Z_{i+1}v\mid
				k\in\mathbb{Z}_{\ge1}\}.
			\end{align}
			From Lemma \ref{lem3} (1), we have
			\begin{align*}
				Z_{i+2+2n}
				Z_{i+1}v
				\in
				T(i+1+n,i+2+n;v),
			\end{align*}
			thus we have
			\begin{align*}
				Z_{i-2n}
				Z_{i+2+2n}
				Z_{i+1}v
				\in
				T(i-2n,i+1+n,i+2+n;v)
				\subseteq{T(i,i,i+3;v)}
			\end{align*}
			for $n\ge1$. Therefore, we obtain
			\begin{align*}
				\SPAN\{Z_{i-2k}Z_{i+2+2k}Z_{i+1}v\mid
				k\in\mathbb{Z}_{\ge1}\}\subseteq{T(i,i,i+3;v)},
			\end{align*}
			and together with (\ref{2i+2,i+1}), we get
			\begin{align} \label{2i+2,i+1,2}
				(Z_{2i+2}(\beta_{2})Z_{i+1}(\alpha_{5})
				-
				e(i)Z_{i}Z_{i+2}
				Z_{i+1})v
				\in
				T(i,i,i+3;v).
			\end{align}
			From Lemma \ref{lem3} (2), we obtain
			\begin{align}\label{i,i+2,i+1}
				(Z_{i}
				Z_{i+2}
				Z_{i+1}
				+\left(\frac{1}{2}\omega^3 + \frac{1}{2}\omega^2 - \frac{1}{2}\right)
				Z_{i}
				Z_{i}
				Z_{i+3})v
				\in
				{T(i,i,i+3;v)}.
			\end{align}
			From (\ref{2i+2,i+1,2})
			and
			(\ref{i,i+2,i+1}),  we have
			\begin{align}\label{2i+2,i+1,3}
				&(Z_{2i+2}(\beta_{2})Z_{i+1}(\alpha_{5})
				+
				e(i)
				\left(\frac{1}{2}\omega^3 + \frac{1}{2}\omega^2 - \frac{1}{2}\right)
				Z_{i}
				Z_{i}
				Z_{i+3})v
				\in
				{T(i,i,i+3;v)}.
			\end{align}


			Moreover, from (\ref{lem1T}), we have
			\begin{align*}
				(Z_{2i}(\beta_{2})
				-d(i)Z_{i}(\alpha_{1})Z_{i}(\alpha_{1}))v
				\in
				T(i,i;v),
			\end{align*}
			thus we obtain
			\begin{align}\label{2i,i+3}
				(Z_{2i}(\beta_{2})Z_{i+3}(\alpha_{5})
				-d(i)Z_{i}Z_{i}
				Z_{i+3})v
				\in
				T(i,i,i+3;v).
			\end{align}
    
    
			From (\ref{middle}),(\ref{i-1,2i+4}),(\ref{2i+2,i+1,3}) and (\ref{2i,i+3}),  we have
			\begin{align}\label{middlecon}
				&(c(\alpha_{5},\beta_{2},1)d(i+2)Z_{i-1}
				Z_{i+2}Z_{i+2}\notag\\*
				&+
				c(\beta_{2},\alpha_{5},1)e(i)
				\left(\frac{1}{2}\omega^3 + \frac{1}{2}\omega^2 - \frac{1}{2}\right)
				Z_{i}
				Z_{i}
				Z_{i+3}\notag\\*
				&-
				c(\beta_{2},\alpha_{5},3)d(i)
				Z_{i}Z_{i}
				Z_{i+3})v
				\in{T(i,i,i+3;v)}.
			\end{align}
			From (\ref{middlecon}) and Lemma \ref{lem4},
			we have
			$Z_{i}
			Z_{i}
			Z_{i+3}v
			\in{T(i,i,i+3;v)}$. 
		\end{proof}

		For $v=v_{\Lambda_{0}+\Lambda_{1}}$,
		we have
		$Z_{-2}Z_{-2}v
		\in{T(-2,-2;v)}$,
		since we obtain
		\begin{align*}
			Z_{-2}Z_{-2}v
			&=(11\omega^5 - 11\omega^3 - 7\omega^2 + 9\omega + 4)
			Z_{-3}Z_{-1}v\notag\\*
			&+(-\frac{1}{18}\omega^5 + \frac{1}{6}\omega^4 
			- \frac{1}{9}\omega^2 - \frac{1}{9}\omega + \frac{1}{9})Z_{-4}
			v
		\end{align*}
		by direct (and tedious) calculation.
				
		The condition $i_{l}\le-2$ (resp. $i_{l}\le-4$)
		in the case of $L(\Lambda_{3})$ (resp. $L(\Lambda_{5})$) follows
		from the principal character $\chi(\Lambda_{3})$ (resp. $\chi(\Lambda_{5})$).
		By the RR type identities,  we can see that the given sets are bases.

\end{document}